\documentclass{article}

\usepackage{arxiv}
\usepackage{lineno,hyperref}
\usepackage[utf8]{inputenc}
\usepackage{amsthm, amsmath, amsfonts, amssymb}
\usepackage[fleqn]{mathtools} 
\DeclarePairedDelimiter\innerproduct{\langle}{\rangle}
\usepackage{csquotes}    
\usepackage{mathrsfs}
\usepackage{xcolor}
\usepackage{verbatim}
\DeclareMathOperator*{\esssup}{ess\,sup}
\newtheorem{lem}{Lemma}[section]
\newtheorem{thm}{Theorem}[section]
\newtheorem{corollary}{Corollary}[thm]

\newtheorem{remark}{Remark}

\newtheorem{definition}{Definition}[section]
\newcommand{\mycomment}[1]{}

\title{Existence and Uniqueness of Solution to Unsteady Darcy-Brinkman Problem with Korteweg Stress for Modelling Miscible Porous Media Flow }

\author{ Sahil Kundu, Surya Narayan Maharana, Manoranjan Mishra \\
  Department of Mathematics, Indian Institute of Technology Ropar, Rupnagar, India 
  }
\date{ 24 May, 2024}

\renewcommand{\undertitle}{}

\begin{document}
\maketitle

\begin{abstract}
	The work investigates a model that combines a convection-diffusion-reaction equation for solute concentration with an unsteady Darcy-Brinkman equation for the flow field, including the Kortweg stress. Additionally, the flow field experiences an external body force term while the permeability fluctuates with solute concentration. Such models are used to describe flows in porous mediums such as fractured karst reservoirs, mineral wool, industrial
foam, coastal mud, etc. The system of equations has Neumann boundary conditions for the solute concentration and no-flow conditions for the velocity field, and the well-posedness of the model is discussed for a wide range of initial data. The proofing techniques remain applicable in establishing the well-posedness of non-reactive and homogeneous porous media flows under the specified simplifications.
\end{abstract}

% keywords can be removed

\section{Introduction}
 Darcy equation is an empirical law proposed by Henry Darcy and is widely used for understanding porous media flows. Several environmental, industrial, and biological entities possess porous structures, and humongous literature investigates Darcy Law's applications when a flow occurs in these structures. CO2 sequestration, groundwater contamination, enhanced oil recovery from natural reservoirs, and tumor growth analysis are chief among them, where the Darcy equation is used for investigations \cite{Huppert2014, Wu2016, Allen98}. Even though Darcy equation-based models enjoy immense geophysical significance (relevant for flows in sinkholes, contaminant transport, and karst aquifers), they are not appropriate for highly porous media with porosity greater than 0.75 \cite{Mccurdy2019}. For example, our globe is filled with chemically active or reactive carbonate rocks (commonly found in karst reservoirs) vulnerable to rock dissolution, creating vugs with pore volumes several orders of magnitude greater than the standard intergranular spaces \cite{Daccord1987,szymczak_ladd_2014,hallack19}. Viscous flows in mediums enriched with such micro and macro length scales are best understood through Brinkman's equation, which is a correction to Darcy's equation that additionally includes Laplacian of the velocity vector \cite{brinkman1949calculation}. The equation is given by $\dfrac{\mu}{K} \boldsymbol{u} = - \boldsymbol{\nabla} p + \mu_{e}\Delta \boldsymbol{u}$, where ``$\boldsymbol{u}$", ``$p$", ``$\mu$", ``$K$" and ``$\mu_{e}$" denote the fluid's velocity, pressure,  fluid's viscosity, the permeability of porous media, and effective viscosity, respectively. The effective viscosity ``$\mu_{e}$" strongly depends both on the porous media structure and the flow field's strength, which can be greater or less than the fluid viscosity ``$\mu$" \cite{krotkiewski2011importance,vafai2005}. 

Steady Darcy-Stokes-based models are extensively used to explain coupled free and porous media flows. For instance, the mixing of surface and ground water \cite{Layton2002}, cancer cell migration \cite{EVJE2017,Evje2018, Qiao19}, platelet aggregation \cite{link2020}, flow in porous-conduit system \cite{Hou16,Babuska10}, etc. These coupled models are comprised of the Darcy equation in one domain while the Stokes equation in the other, connected through proper mass and force balance laws like Beavers--Joseph (BJ) interfacial condition (or some variants of it) to describe the flow; see \cite{chen2010asymptotic, HOU2019, cai2009, Cao2011, FENG2012453, Girault2009} and references therein. Numerical handling of such a two-domain problem is highly non-trivial, and the computational cost is huge. Thanks to homogenization theory that paved the way out and established the theory that two-domain problems with complex interfacial conditions can be avoided by solving Brinkman's equation in a single domain \cite{Allaire1991, Greibel2010, brown2015, Cesmelioglu2009}. Indeed, the steady Brinkman equation is used to analyze non-Newtonian fluids \cite{BULICEK2015109},  flows in mineral wool and industrial textiles/foam \cite{Iliev11}, local and non-local growth of tumors in \cite{Ebenbeck2019,1937-1632_2021_11_3989}, chromatography separation \cite{CARTA199262}, etc. Moreover, with proper non-dimensionalization of the Brinkman model, Stokes and Darcy's problems are recovered when $\mu_{e} = 1$ (or infinite permeability, $K \rightarrow \infty$ ) and $\mu_{e} = 0$, respectively \cite{mu2020, Badia2009}.

Although the inclusion of the time derivative of velocity is debated for being negligible compared to other terms \cite{Mccurdy2019, Vafai2015}, the unsteady Darcy-Brinkman equation has been used and scrutinized by several numerical analysts \cite{Hou16,John2015,keim2016,Layton2013}. The applications of unsteady Darcy-Brinkman equation based models can be realised in biological processes and petroleum extraction in karst reservoirs \cite{krotkiewski2011importance,fritz2019unsteady, ligaarden2010stokes,popov2009multiphysics}. In 1901, Korteweg proposed that a body force caused by a concentration ( or density ) could act like an interfacial tension \cite{Korteweg1901}. Also experimentally, it has been found that steep gradient in concentration profiles leads to interfacial tension \cite{Pojman2006}. This force or tension is known as Korteweg stress, and it can control hydrodynamic instability such as viscous fingering \cite{SWERNATH20102284,mishra_2022,deki2023numerical} and convection in miscible polymer-monomer systems \cite{Pojman2009}. In this paper, using the Galerkin method, we prove the existence and uniqueness of weak solutions for problems that use the unsteady Darcy-Brinkman equation as a flow equation with Korteweg stress coupled to transport equations accounting for fluid's miscibility or reactive nature. Thus, the techniques used in the current article are also valid for purely unsteady Darcian or Stokes flow with the aforementioned limiting cases.

Fluid's miscible transport is commonly modeled by coupling a convection-diffusion equation
representing the conservation of solute concentration with the flow equations (conservation laws of mass and momentum). Reactive transport modeling is done through a change in the convection-diffusion equation to a convection-diffusion-reaction system for conserving concentration representatives of underlying reactive species. See \cite{anne2020} for various kinds of reactions in the literature and their chemo-hydrodynamic effects on pattern modulations. This article also examines solutions' existence and uniqueness when the chemical species undergoes a simple first-order reaction. Moreover, this reactive flow is analyzed under a forced flow field \cite{allali2015model, allali2017existence, migorski2019nonmonotone}, where an external body force term is considered in the Darcy-Brinkman equation. The body force can be realised as gravity or bouncy force in practicality, which comes to model by Boussinesq approximation and exhibits massive significance in the context of geophysical flows \cite{smyth_carpenter_2019}.

The present study focuses on modeling in a porous media flow with high porosity, specifically reactive karst reservoirs, by considering the Darcy-Brinkman equation as the governing equation. In such reservoirs, invading fluids react with the solid phase, which increases the porosity of the medium. Porosity can be a function of volume fraction or fluid concentration, which presents a problem in investigating the existence and uniqueness of solutions when the porosity becomes a function of concentration in the Darcy-Brinkman equation.

Experimental observations have shown that a precipitation reaction can increase the porosity of the medium, thus affecting convection and hydrodynamic instabilities \cite{White2012,Nagatsu2014}. The authors of \cite{Nagatsu2014} have validated their experimental results by using the numeric, taking the permeability as an exponentially decaying function of the product fluid's concentration $C$, i.e., $K(C)= e^{-2 R_{c} C}$, where $R_{c}$ is a non-negative number. In this article, we examine the uniqueness and existence of solutions to the generalized version of the non-reactive model, where the permeability is considered to be a reciprocal of the linear function of the solute's concentration.

The rest of the paper is structured as follows: Section \ref{sec2} introduces functional spaces and preliminary results that will be used throughout the article. Section \ref{sec:math_model} presents the fluid flow model, which consists of the Brinkman equation with an imposed external body force term coupled to convection-diffusion-reaction (CDR) and continuity equations. The fluid is assumed to be incompressible, Newtonian, and of constant viscosity, while the porous media is assumed to be heterogeneous. This heterogeneity is introduced through precipitation, where the permeability becomes a function of ``$C$" rather than a constant. The weak formulation of the model is also defined in this section.
Section \ref{sec4} focuses on obtaining priori estimates for the velocity vector ``$\boldsymbol{u}$" and solute concentration ``$C$", and proving the existence and uniqueness of the model presented in section \ref{sec:math_model}.

\section{Notations and Auxiliary results}\label{sec2} In this section, we recall some standard notations and existing important results that will be used frequently throughout the paper. We denote $(\cdot, \cdot)_{X}$ as the inner product for a given inner product space $X$. For a Banach space $X$ and its dual space $X^{*}$, the duality pairing is denoted by $\langle\cdot, \cdot\rangle_{X}$. For notational convenience, we denote the duality pairing by $\innerproduct{.,.}$ for $X = \left(H_{0}^{1}\right)^2$. We define the scalar product of two matrices by $$\mathbf{A}:\mathbf{B} =  \sum_{i,j=1}^{n} a_{ij}b_{ij} ~ ~ \text{ for } \mathbf{A}, \mathbf{B} \in \mathbb{R}^n . $$ 

Let $\Omega$ denote an open bounded subset of $\mathbb{R}^2$  with $\mathcal{C}^{1}$ boundary $\partial \Omega $. Usually, the standard Sobolev spaces for $k\in \mathbb{N}$ are defined as,
$$ W^{k,p} = \left\{\phi \in L^{p}(\Omega): D^{\alpha}\phi \in L^{p}(\Omega) \hspace{5pt} \forall |\alpha| \leq k \right\},$$
where $ D^{\alpha}f = \dfrac{\partial^{\alpha}f}{\partial x^{\alpha_{1}}_{1}....\partial x^{\alpha_{n}}_{n}}$ denotes mixed partial derivative of order $\alpha$ in the distributional sense. The space ``$ W^{k,p}$" is equipped with the following norm,
$$
\|\phi\|_{W^{k, p}(\Omega)}:= \begin{cases}\sum_{|\alpha| \leq k}\left\|D^{\alpha} \phi \right\|_{L^{p}(\Omega)} & 1 \leq p<\infty ; \\ \max_{|\alpha| \leq k}\left\|D^{\alpha} \phi \right\|_{L^{\infty}(\Omega)} & p=\infty .\end{cases}
$$
For the fixed value $p = 2$, we denote $ W^{2,k}(\Omega)$ by $H^{k}(\Omega)$. So the norm $
\|\cdot\|_{W^{k, p}(\Omega)}$ = $\|\cdot\|_{H^{k}(\Omega)}$. The elements of $H^{k}(\Omega)$ that vanish at the boundary $\partial \Omega$ are in the space $H^{k}_{0}(\Omega)$, i.e., $H^{k}_{0}(\Omega)=\{ \phi \in H^{k}(\Omega) : \phi=0 \hspace{.2 cm} \text{on} \hspace{.2 cm} \partial  \Omega \} $. We denote $H^{-k}(\Omega)$  as the dual space of $H^{k}_{0}(\Omega)$.
We further define the spaces $ S_{1}, V_{1}, V_{2}  $ as follows:
\begin{center}
$\begin{aligned}
{S_{1}} &= \left\{\boldsymbol{v} \in \left(L^{2}(\Omega)\right)^{2} : \boldsymbol{\nabla} \cdot \boldsymbol{v}=0 , (\boldsymbol{v}\cdot \boldsymbol{\eta})|_{\partial \Omega} = 0    \right\}, \\ 
{V_{1}} &= \left\{\boldsymbol{v} \in\left(H^{1}_{0}(\Omega)\right)^{2} : \boldsymbol{\nabla} \cdot \boldsymbol{v} = 0 \right\}, \\
V_{2} &= \left\{ B \in H^{2}(\Omega) : \dfrac{\partial B}{\partial \boldsymbol{\eta}}\bigg|_{\partial \Omega} = 0  \right\}.
\end{aligned}$
\end{center}
Here ``$ \boldsymbol{\eta} $" denotes the outward unit normal vector. For the velocity vector $\boldsymbol{v} \in S_{1}$ or $V_{1}$, the divergence is meant in the distributional sense.  
We further let
$$ {L^{p}(a, b ; Y)} = \left\{ \phi:[a, b] \rightarrow Y : \|\phi\|_{L^{p}(a, b ; Y)}=\left(\int_{a}^{b}\|\phi\|_{Y}^{p} \, \mathrm{d} t \right)^{\frac{1}{p}}  < \infty      \right\} ~ \text{for} ~1 \leq p < \infty, $$ 
and 
$${L^{\infty}(a, b ; Y)} = \left\{ \phi:[a, b] \rightarrow Y :  \|\phi\|_{L^{\infty}(a, b ; Y)}= \esssup_{[a, b]}\|\phi\|_{Y} < \infty      \right\},$$
where $Y$ is given as a Banach space equipped with the norm $\| \cdot \|_{Y}$.
Let $S_{\boldsymbol{u}} = L^{2}(0,T;V_{1}) $ and  $S_{C} = L^{2}(0,T;V_{2})$. The space, $\mathcal{C}([0,T];X)$ contains all the continuous functions $\phi : [0,T] \rightarrow X$ with $\| \phi\|_{ \mathcal{C}([0,T]};X)  = \max_{0\leq t \leq T} \|\phi(t)\|_{X} < \infty$.
we will use $\left(. , . \right)$ to denote the standard inner product in $L^2(\Omega)$ as well as in $(L^2(\Omega))^2$.
Throughout this paper, we will use $M< \infty$ as a generic constant.

\begin{thm}[Grönwall's inequality, cf. {\cite[Lemma~3.1]{fritz2019unsteady}}] \label{GR}
Let $f, g \in \mathcal{C}\left([0, T] ; \mathbb{R}_{\geq 0}\right)$. If there are constants $0 \leq C_1, C_2<\infty$ such that
$$
f(t)+g(t) \leq C_1+C_2 \int_0^t f(s) \mathrm{d} s \quad \text { for all } t \in[0, T],
$$
then it holds that $f(t)+g(t) \leq C_1 e^{C_2 T}$ for all $t \in [0, T]$.
\begin{remark}
A more generalized version of Grönwall's inequality can be seen in the appendix of \cite{evans2022partial}, where $f$ is only considered to be a nonnegative summable function on [0, T].

\end{remark}
\end{thm}
\begin{thm}[Gagliardo–Nirenberg, cf. { \cite[Lemma~1]{migorski2019nonmonotone}}, { \cite[Lemma~1.1]{garcke2019}}]\label{gagliardo} 
If $\Omega \subset \mathbb{R}^n \,(n=2,3)$ is a domain with $\mathcal{C}^{1}$ boundary, then there exists a constant $M>0$ depending only on $\Omega$ such that, in the case $n=2$:
$$
\|\boldsymbol{\phi}\|_{L^{4}} \leq M\|\boldsymbol{\phi}\|_{L^{2}}^{1 / 2}\|\boldsymbol{\nabla} \boldsymbol{\phi}\|_{L^{2}}^{1 / 2} \hspace{5pt};\hspace{5pt} \forall \hspace{5pt} \boldsymbol{\phi} \in H^{1}(\Omega)
$$
and in the case $n=3$:
$$
\|\boldsymbol{\phi}\|_{L^{4}} \leq M\|\boldsymbol{\phi}\|_{L^{2}}^{1 / 4}\|\boldsymbol{\nabla} \boldsymbol{\phi}\|_{L^{2}}^{3 / 4}\hspace{5pt};\hspace{5pt} \forall \hspace{5pt} \boldsymbol{\phi} \in H^{1}(\Omega).
$$  
\end{thm}
\begin{thm}[Aubin-Lions, cf. {\cite[Lemma~3]{migorski2019nonmonotone}} ]\label{aubin-lions}
Let $X_{1}, X_{2}$ and $X_{3}$ be reflexive Banach spaces.  If $X_{1} \subset X_{2} \subset X_{3}$ 
continuously with compact embedding $X_{1} \subset X_{2}$ and $p, q \in(1, \infty)$,  then for any $T>0$, the embedding of the space
$$
\left\{\phi \in L^{p}(0, T ; X_{1}) \mid \dfrac{\partial \phi}{\partial t} \in L^{q}(0, T ; X_{3})\right\}
$$ into $L^{p}(0, T ; X_{2})$ is compact.
\end{thm}
\begin{remark}
 If $p=\infty$ and $q>1$, then the space $\left\{\phi \in L^{p}(0, T ; X_{1}) \mid \dfrac{\partial \phi}{\partial t} \in L^{q}(0, T ; X_{3})\right\}$ is compactly embedded into $\mathcal{C}([0, T] ; X_{2})$.
\end{remark}
\begin{thm}[de Rham, cf. {\cite[Theorem IV.2.4]{boyer2012mathematical}}] \label{de rham}
Let $\Omega$ be a connected, bounded, Lipschitz domain of $\mathbb{R}^d$. Let $\boldsymbol{h}$ be an element in $\left(H^{-1}(\Omega)\right)^d$, such that for any function $\varphi \in(\mathcal{D}(\Omega))^d$ satisfying $\operatorname{div} \varphi=0$, we have $\innerproduct{\boldsymbol{h},\varphi}_{(H_0^1)^d}=0$. Then, there exists a unique function $p$ belonging to $L_0^2(\Omega) = \left(\phi \in L^2 : \int_{\Omega} \phi = 0\right)$ such that $\boldsymbol{h}=\nabla p$.
\end{thm}
\noindent\textbf{Young's inequality.}
Let $p, q$ be positive real numbers satisfying $\frac{1}{p}+\frac{1}{q}=1$. If $a, b$ are non-negative real numbers, then
$$
a b \leq \frac{a^p}{p}+\frac{b^q}{q},
$$
and the equality holds if and only if $a^p=b^q$.
\begin{remark}\label{youngs}
We will use another form of Young's inequality in some of our proofing techniques which states that for any $\epsilon>0$ there exists a constant $A_{\epsilon}>0$ such that, 
$$
a b \leq \epsilon a^{p} + A_{\epsilon}b^{q}.
$$
\end{remark}
\section{Mathematical Model}\label{sec:math_model}
In this section, we introduce the unsteady Darcy-Brinkman model, which describes miscible fluid flow in a heterogeneous porous medium in the presence of the Korteweg stress, external body force, and the reaction term. To account for the heterogeneity of the porous media, we assume the permeability of the porous medium as a function of concentration.

The fluid under consideration is assumed to be incompressible, neutrally buoyant and miscible in nature. Let the bounded domain of the flow be  $\Omega \subset \mathbb{R}^2$ and (0, T) be the time interval. The following system of equations governs the flow:
\begin{equation}\label{1}\boldsymbol{\nabla} \cdot \boldsymbol{u}=0 \quad \text{in}\ (0,T) \times \Omega,  \end{equation}
\begin{equation}\label{2}
\dfrac{\partial \boldsymbol{u}}{\partial t}+\dfrac{\mu}{K(C)}\boldsymbol{u}=-\boldsymbol{\nabla} p+\mu_{e} \Delta \boldsymbol{u}+\boldsymbol{\nabla} \cdot \mathbf{T}\left( C\right) + \boldsymbol{f}\quad \text{in}\ (0,T) \times \Omega,
\end{equation}
\begin{equation}\label{3}
\dfrac{\partial C}{\partial t}+\boldsymbol{u}\cdot \boldsymbol{\nabla} C=d\Delta C -C g\quad \text{in}\ (0,T) \times \Omega. \end{equation}
Here, $\boldsymbol{u}$ represents the fluid velocity, $p$ denotes pressure, and $C$ denotes the solute concentration in the fluid. The diffusion coefficient $d$, viscosity $\mu$, and effective viscosity $\mu_{e}$ are assumed to be strictly positive constants. The equation \eqref{1} corresponds to the law of mass conservation and the conservation of momentum is described through the equation \eqref{2}, the unsteady Darcy-Stokes equation. These flow equations are coupled to a convection-diffusion equation for the solute concentration $C$ to describe the miscible displacement of the fluid. The coupling of the flow equations to the solute tracer has occurred through the Kortweg Stress `$\mathbf{T}(C)$' and by permeability  $``K(C) = (\alpha+\beta C)^{-1}$", where $\alpha>0$ and $\beta \geq 0$ are constants. The Korteweg stress tensor \cite{pramanik2013linear} `$\mathbf{T}(C)$' is given by $$\mathbf{T}(C)=\left(-\frac{1}{3} \hat{\delta}|\boldsymbol{\nabla} C|^{2}+\frac{2}{3} \gamma \boldsymbol{\nabla}^{2} C\right) I+\hat{\delta} \boldsymbol{\nabla}  C \otimes \boldsymbol{\nabla} C, $$
where $\otimes$ denotes the tensor product and $I$ is the identity matrix. The two parameters $\gamma > 0$ and $\hat{\delta} > 0$ are known as the Korteweg parameters. Following the continuum theory of surface tension, we can write \cite{JOSEPH1996104}
$$\boldsymbol{\nabla} \cdot \mathbf{T}(C)=\boldsymbol{\nabla} Q(C)-\hat{\delta} \boldsymbol{\nabla} \cdot[\boldsymbol{\nabla} C \otimes \boldsymbol{\nabla} C],
$$
where Q(C) is the term that appears as the coefficient of the identity matrix $I$ in the stress term $\mathbf{T}(C)$. We simplify $\nabla \cdot \mathbf{T}(C)$ by expanding last term in above expression
$$\boldsymbol{\nabla} \cdot \mathbf{T}(C) =\boldsymbol{\nabla} Q(C)-\hat{\delta}(\boldsymbol{\nabla} C \boldsymbol{\nabla} \cdot \boldsymbol{\nabla} C + (\boldsymbol{\nabla} C \cdot \boldsymbol{\nabla})\boldsymbol{\nabla} C) = \boldsymbol{\nabla} Q(C)-\dfrac{\hat{\delta}}{2} \boldsymbol{\nabla} [(\boldsymbol{\nabla} C)^2]-\hat{\delta}\boldsymbol{\nabla} C \Delta C,
$$ 
\begin{equation*}
\Longrightarrow ~ \boldsymbol{\nabla} \cdot \mathbf{T}(C)= \boldsymbol{\nabla} \left(Q(C)-\dfrac{\hat{\delta}}{2}(\boldsymbol{\nabla} C)^2\right)-\hat{\delta}\boldsymbol{\nabla} C \Delta C.
\end{equation*}
 Note that if $C \in L^2(0,T;V_2)$, then we have  
\begin{align}\label{4}
 \innerproduct*{\boldsymbol{\nabla} \cdot \mathbf{T}(C(t)),\boldsymbol{v}} &= \innerproduct*{\boldsymbol{\nabla} \left(Q(C(t))-\dfrac{\hat{\delta}}{2}(\boldsymbol{\nabla} C(t))^2\right)-\hat{\delta}\boldsymbol{\nabla} C(t) \Delta C(t),\boldsymbol{v}}\nonumber\\ & = \innerproduct*{\boldsymbol{\nabla} \left(Q(C(t))-\dfrac{\hat{\delta}}{2}(\boldsymbol{\nabla} C(t))^2\right),\boldsymbol{v}}  + \innerproduct{-\hat{\delta}\boldsymbol{\nabla} C(t) \Delta C(t),\boldsymbol{v}}\nonumber\\ & = -\hat {\delta} \big(\boldsymbol{\nabla} C(t) \Delta C(t) , \boldsymbol{v}\big)~ ~ a.  e. ~ on ~ (0,T),~ \forall \boldsymbol{v} \in V_1.
\end{align}
The boundary of a reservoir or karstified region can typically be considered impermeable. So the system \eqref{1}-\eqref{3} is supplemented with the following no-slip, no-penetration and no-flux boundary conditions: \begin{equation}\label{5}
\boldsymbol{u} = 0;\dfrac{\partial C}{\partial \boldsymbol{\eta} }=0 \hspace{5pt} \text{on} \hspace{5pt}  (0,T) \times \partial\Omega, \end{equation}
where $\boldsymbol{\eta}$ denotes unit outward normal vector to $\partial\Omega$. The initial conditions for the problem \eqref{1}-\eqref{3} are given as
\begin{equation}\label{6} \boldsymbol{u}\left(\boldsymbol{x},0\right) =\boldsymbol{u}_{0}\left( \boldsymbol{x}\right) ;C\left( \boldsymbol{x},0\right) =C_{0}\left(\boldsymbol{x}\right) \hspace{5pt} \text{ for } \boldsymbol{x} \in \Omega,
\end{equation}
where $\boldsymbol{u}_{0}$ and $C_{0}$ are given  functions of spatial variable only. We assume that the external body force term $\boldsymbol{f}(\boldsymbol{x},t)$ belongs to $L^{2}(0,T;L^2(\Omega))$, for $T>0$ and that $g(\boldsymbol{x})$ in the reaction term is non-negative and belongs to  $L^{\infty}(\Omega)$. Additionally, we assume a constraint on pressure $\int_\Omega p  ~dx = 0$ to ensure the uniqueness of the solution.

%%%%%%%%%%% Varitional form %%%%%%%%%%%%%%%%%%%%%%%%%%%%%%%%%%%%%%%%%
%%%%%%%%%%%%%%%%%%%%%%%%%%%%%%%%%%%%%%%%%%%%%%%%%%%%%%%%%%%%%%%%%%%%%%%
\subsection{Variational Form}
 By multiplying equations \eqref{2} and \eqref{3} with appropriate test functions and applying integration by parts while considering the boundary conditions, we can derive the variational form for the problem defined by equations \eqref{1}-\eqref{6}. The variational form corresponding to the problem described in equations \eqref{1}-\eqref{6} is as follows:
 \begin{equation}\label{9}
 \left (\dfrac{\partial C(t)}{\partial t},B\right ) +d (\boldsymbol{\nabla} C(t) , \boldsymbol{\nabla} B) + (\boldsymbol{u}(t)\cdot \boldsymbol{\nabla} C(t),B)+(g C(t),B) =0 , ~ \forall B \in V_{2}  ~ ~  a.  e. ~ on ~ (0,T),
\end{equation}
\begin{equation}\label{10}
\innerproduct*{\dfrac{\partial \boldsymbol{u}(t)}{\partial t},\boldsymbol{v}} +\mu_{e}(\boldsymbol{\nabla} \boldsymbol{u}(t) , \boldsymbol{\nabla} \boldsymbol{v}) +\mu \big((\alpha + \beta C(t))\boldsymbol{u}(t), \boldsymbol{v}\big) - \innerproduct*{\boldsymbol{\nabla} \cdot \mathbf{T}\left( C(t)\right), \boldsymbol{v}} - (\boldsymbol{f}(t),\boldsymbol{v}) = 0 , ~ \forall \boldsymbol{v} \in V_{1}  ~ ~  a.  e. ~ on ~ (0,T).
\end{equation}

\begin{definition}[weak solution for the problem \eqref{1}-\eqref{6}]\label{def31}
We call a pair $(\boldsymbol{u}, C)  \in \Big(L^{2}(0,T;V_{1}), L^{2}(0,T;V_{2})\Big)$, $ \forall\, T>0$ such that $ \boldsymbol{u}(0) =\boldsymbol{u}_{0} ;\,C(0) =C_{0}\hspace{5pt} \, a.e.\ \text{in}\ \Omega$, is a weak solution to the problem \eqref{1}-\eqref{6}, if this pair satisfies equation \eqref{9} and \eqref{10}.
\end{definition}

%% Existence and Uniqueness%%%%%%%%%%%%%%%%%%%%%%%%%%%%%%%%%%%%%%%%%%%%%%%%%%%%%%%%%%%%%%%%
%%%%%%%%%%%%%%%%%%%%%%%%%%%%%%%%%%%%%%%%%%%%%%%%%%%%%%%

\section{Existence and Uniqueness}\label{sec4}
In this section, we show the existence and uniqueness of the solution of the problem \eqref{1}-\eqref{6}.
\begin{thm}\label{th41}
For any initial condition $ (\boldsymbol{u}_{0},C_{0}) \in (S_{1},H^1) $, there exists a solution pair $(\boldsymbol{u},C)$ of the problem \eqref{1}-\eqref{6} in the sense of Definition \ref{def31}. Furthermore, we have
	$ \boldsymbol{u} \in L^{2}(0,T;V_{1}) \cap  \mathcal{C}([0,T];S_{1})$  and $ C \in  L^{2}(0,T;V_{2}) \cap \mathcal{C}([0,T];H^1). $
\end{thm}
To establish the existence of a solution, we employ the Galerkin method. We use a special basis of $V_{2}$ composed with eigenvectors of the negative of the Laplace operator associated with the Neumann boundary condition corresponding to eigen values  $\{\lambda_j\}_{j=1}^{\infty}$ and an arbitrary basis of $V_{1}$. Let the basis for  $V_{1}$ be given by ($\boldsymbol{w}_1,\boldsymbol{w}_2,.....$) and let the basis for $V_{2}$ be given by ($z_1,z_2,.....$). Let $(V_1)_n = span(\boldsymbol{w}_1,\boldsymbol{w}_2,.....,\boldsymbol{w}_n) $
and $(V_2)_n = span(z_1,z_2,.....z_n) $ are finite dimensional subspaces of $V_1$ and $V_2$, respectively.
Now we look for the functions $\boldsymbol{u}_{n}:[0,T] \rightarrow (V_1)_n$ and $C_{n}:[0,T] \rightarrow (V_2)_n$ of the form
$$\boldsymbol{u}_n = \sum_{j=1}^{n} \alpha_{j}^{n}(t) \boldsymbol{w}_{j} ~ \text{ and } ~ C_{n} = \sum_{j=1}^{n} \beta_{j}^{n}(t) z_{j},$$
so that these functions satisfy the following equations
\begin{align}
\label{15}
 \left (\dfrac{\partial C_{n}(t)}{\partial t},z_{j}\right ) +d (\boldsymbol{\nabla} C_{n}(t) , \boldsymbol{\nabla} z_{j}) + (\boldsymbol{u}_n(t)\cdot \boldsymbol{\nabla} C_{n}(t),z_{j}) + (gC_n(t),z_{j})=0 ~, ~z_{j} \in (V_{2})_n ,    
\end{align}
\begin{align}\label{16}
\innerproduct*{\dfrac{\partial \boldsymbol{u}_{n}(t)}{\partial t},\boldsymbol{w}_j} +\mu_{e}(\boldsymbol{\nabla} \boldsymbol{u}_{n}(t) , \boldsymbol{\nabla} \boldsymbol{w}_j) + \mu\left( \left( \alpha + \beta C_{n}(t) \right)\boldsymbol{u}_{n}(t), \boldsymbol{w}_j\right) -\innerproduct{\boldsymbol{\nabla} \cdot \mathbf{T}\left( C_{n}(t)\right), \boldsymbol{w}_j} - (\boldsymbol{f}(t),\boldsymbol{w}_j)= 0 ~,~ \boldsymbol{w}_j \in (V_{1})_n ,   
\end{align}
$ a.  e. ~ on ~ (0,T)$, for $j \in \{1,2,3....n \}$, and satisfy the initial conditions
$$ \boldsymbol{u}_{n}(0) = \sum_{j=1}^{n} \left( \boldsymbol{u}_{0},\boldsymbol{w}_{j}\right) \boldsymbol{w}_{j} ~ \text{ and } ~ C_{n}(0) = \sum_{j=1}^{n}\left(C_{0},z_{j}\right)z_{j}.  $$
Since $(V_1)_n$ and $(V_2)_n$ are finite dimensional, equations \eqref{15} and \eqref{16} correspond to a nonlinear system of first-order ordinary differential equations with constant coefficients. Thus, the local existence of $\boldsymbol{u}_n$ and $C_n$ is proved by Carathéodory's existence theorem. 

Now we prove some priori estimates for the velocity field sequence $\boldsymbol{u}_n$, concentration sequence $C_n$, and their derivatives. Then, we use these priori estimates to employ the continuation argument to extend our solutions for each $t \in (0, T], \text{ for given } T>0$.

%%%%%% Lemma 4.1 %%%%%%%%%%%%%%%%%%%%%%%%%%%%%%%%%%%%%%%%%%%%%%%%%%%%%%%%%%%%%%%%%%%%%%%%%%%%
%%%%%%%%%%%%%%%%%%%%%%%%%%%%%%%%%%%%%%%%%%

\begin{lem}\label{le41}
 The concentration sequence $C_n$ is uniformly bounded in both $L^{\infty}(0,T;L^{2}) $ and $L^{2}(0,T;H^{1})$ spaces. 
 \end{lem}
\begin{proof}  
By multiplying equation \eqref{15}, with $\beta_{j}^{n}$ and summing the resulting equations over $j=1$ to $n$, we get 
\begin{align*}
\left (\dfrac{\partial C_n(t)}{\partial t},C_n(t)\right ) +d (\boldsymbol{\nabla} C_n(t) , \boldsymbol{\nabla} C_n(t)) + (\boldsymbol{u}_n\cdot \boldsymbol{\nabla} C_n(t),C_n(t))+(g C_n(t),C_n(t)) = 0 ~ ~ \text{ for } a.  e. ~ on ~ (0,T).   
\end{align*}
\noindent  As we have,  $g\geq 0$, so $  (g C_n, C_n) \geq 0$. Therefore
\begin{align*}
\dfrac{1}{2}\dfrac{d}{dt}\left(C_n(t),C_n(t)\right)+d\left(\boldsymbol{\nabla} C_n(t),\boldsymbol{\nabla} C_n(t)\right) +\int_{\Omega}\boldsymbol{u}_n(t)\cdot \dfrac{1}{2}\boldsymbol{\nabla} \left( C_n^{2}(t)\right) \leq 0 ~ ~a.  e. ~ on ~ (0,T).
\end{align*}
Using integration by parts on the last term of the left-hand side of the above equation, we have \begin{align*}
 \dfrac{1}{2}\dfrac{d}{dt}\left\| C_n(t)\right\| _{L^{2}\left(\Omega \right) }^{2}+d{\left\| \boldsymbol{\nabla} C_n(t)\right\|}^{2} _{L^{2}\left( \Omega\right)}+\dfrac{1}{2}\left[\int _{\partial \Omega}C_n^{2}(t)\boldsymbol{u}_n(t)\cdot \boldsymbol{\eta} -\int_{\Omega} C_n^{2}(t)\boldsymbol{\nabla} \cdot  \boldsymbol{u}_n(t)\right]  \leq 0  ~ ~  a.  e. ~ on ~ (0,T).  \end{align*}
From boundary conditions \eqref{5} and the continuity equation \eqref{4}, we observe that the last term of the above equation vanishes. So we arrive at,
\begin{align*}
 \dfrac{1}{2} \dfrac{d}{dt}\left\| C_n(t)\right\| _{L^{2}\left(\Omega \right) }^{2}+d{\left\| \boldsymbol{\nabla} C_n(t)\right\|}^{2} _{L^{2}\left( \Omega\right) } \leq 0  ~ ~  a.  e. ~ on ~ (0,T).   
\end{align*}
Integrating the above equality in time from $0$ to $\tau \in (0,T]$, we obtain
\begin{align*}
\left\| C_n\left(\tau\right)\right\| _{L^{2}\left(\Omega \right) }^{2} + 2 d \int_{0}^{\tau}{\left\| \boldsymbol{\nabla} C_n(t)\right\|}^{2} _{L^{2}\left( \Omega\right) } \leq \left\| C_n\left(0\right)\right\| _{L^{2}\left(\Omega \right) }^{2}\leq \left\| C(0)\right\| _{L^{2}(\Omega)}^2;\hspace{5pt} \forall \tau > 0. \hspace{5pt}    
\end{align*}
So, $\esssup _{\tau \in \left[ 0,T\right] }\left\| C_n\left( \tau \right) \right\| _{L^{2}}\leq \left\| C(0)\right\| _{L^{2}} .$
Hence the concentration sequence $C_n$ is uniformly bounded in both\\ $L^{\infty}(0,T;L^{2})$ and $L^{2}(0,T;H^1)$ spaces.
\end{proof}

%%%%%%%%%%%%%%%%%%%%%%%%%%%% Lemma 2 %%%%%%%%%%%%%%%%%%%%%%%%%%%%%%%%%%%%%%%%%%%%%%%%%%%%%
%%%%%%%%%%%%%%%%%%%%%%%%%%%%%%%%%%%%%%%%%%%%%%%%%%%%%%%
\begin{lem}\label{le42}
The velocity field sequence $\boldsymbol{u}_n$ is uniformly bounded in both $L^{\infty}(0,T;S_{1})$ and  $L^{2}(0,T;V_{1})$ spaces. The concentration sequence $C_n$ is uniformly bounded in both $ L^{\infty}(0,T;H^1)$ and   $ L^{2}(0,T;V_{2})$ spaces. \end{lem}
\begin{proof}
 By multiplying equation \eqref{15} with $\hat \delta \lambda_{j}\beta_{j}^{n}$, and summing the resulting equations over  $j = 1$ to $n$, and utilizing the fact that $-\Delta C_n = \lambda_j C_n$, we conclude
 \begin{align*}
\left(\dfrac{\partial C_n(t)}{\partial t}, -\hat \delta\Delta C_n(t)\right) + \left(\boldsymbol{u}_n(t)\cdot \boldsymbol{\nabla} C_n(t),  -\hat \delta\Delta C_n(t)\right) - d\left(\Delta C_n(t), -\hat \delta\Delta C_n(t)\right)+ \left(C_n(t)g, -\hat \delta\Delta C_n(t)\right)  = 0,
\end{align*}
$ a.  e. ~ on ~ (0,T)$.
Using integration by parts and  boundary conditions \eqref{5}, we arrive at
\begin{align*}
\dfrac{\hat \delta}{2}\dfrac{d}{dt} \left(\boldsymbol{\nabla} C_n(t), \boldsymbol{\nabla} C_n(t)\right) + d\hat \delta \left(\Delta C_n(t), \Delta C_n(t)\right) = \hat \delta \left(\boldsymbol{u}_n(t)\cdot \boldsymbol{\nabla} C_n(t),\Delta C_n(t)\right) + \hat \delta \left(C_n(t)g, \Delta C_n(t)\right) ~ ~ a.  e. ~ on ~ (0,T).   
\end{align*}
Using Holder's inequality on the last term on the right-hand side of the above equation, we have
\begin{align*}
\dfrac{\hat \delta}{2}\dfrac{d}{dt}\left\| \boldsymbol{\nabla} C_n(t)\right\| _{L^{2}}^{2}+d\hat \delta \left\| \Delta C_n(t)\right\| _{L^{2}}^{2}\leq \hat \delta \left(\boldsymbol{u}_n(t)\cdot \boldsymbol{\nabla} C_n(t),\Delta C_n(t)\right) +  \hat  \delta \|g C_n(t)\|_{L^2} \| \Delta C_n(t)\|_{L^2}  ~ ~ a.  e. ~ on ~ (0,T). 
\end{align*}
Using Young's inequality (Remark \ref{youngs}), and the assumption that $g \in L^{\infty}(\Omega)$, we conclude 
\begin{align}\label{11}
 \dfrac{\hat \delta}{2}\dfrac{d}{dt}\left\| \boldsymbol{\nabla} C_n(t)\right\|_{L^{2}}^{2}+d\hat \delta \left\| \Delta C_n(t)\right\|_{L^{2}}^{2}&\leq \hat \delta \left(\boldsymbol{u}_n(t)\cdot \boldsymbol{\nabla} C_n(t),\Delta C_n(t)\right) +\hat \delta \left(\epsilon\left\| \Delta C_n(t)\right\| _{L^{2}}^{2} + M(\epsilon)\left\| g\right\|_{L^{\infty}(\Omega)}^2\left\|  C_n(t)\right\|_{L^{2}}^{2}\right)   \end{align}
$a.  e. ~ on ~ (0, T)$. By multiplying equation \eqref{16} with $\alpha_{j}^{n}$, and summing over $j=1$ to $n$, we get
\begin{align*}
 \dfrac{1}{2}\dfrac{d}{dt}\left\| \boldsymbol{u}_n(t)\right\| _{L^{2}}^{2}+\mu_{e}\left\| \boldsymbol{\nabla} \boldsymbol{u}_n(t)\right\| _{L^{2}}^{2}+\alpha \mu \left\| \boldsymbol{u}_n(t)\right\| _{L^{2}}^{2} = \innerproduct{\boldsymbol{\nabla} \cdot \mathbf{T}\left( C_n(t)\right), \boldsymbol{u}_n(t)} - \beta \mu \big(C_n(t) \boldsymbol{u}_n(t),\boldsymbol{u}_n(t)\big) + \big(\boldsymbol{f}(t),\boldsymbol{u}_n(t)\big), 
\end{align*}
$ a.  e. ~ on ~ (0,T)$. Using expression of $\innerproduct{\boldsymbol{\nabla} \cdot \mathbf{T}\left( C_n(t)\right), \boldsymbol{u}_n(t)}$, from equation \eqref{4}, the above equation further simplifies to
\begin{align}\label{12}
\dfrac{1}{2}\dfrac{d}{dt}\left\| \boldsymbol{u}_n(t)\right\| _{L^{2}}^{2}+\mu_{e} \left\| \boldsymbol{\nabla} \boldsymbol{u}_n(t)\right\| _{L^{2}}^{2}+\alpha \mu \left\| \boldsymbol{u}_n(t)\right\| _{L^{2}}^{2} = -\hat \delta\ \big(\Delta C_n(t)\boldsymbol{\nabla} C_n(t) , \boldsymbol{u}_n(t)\big) - \beta \mu \big(C_n(t) \boldsymbol{u}_n(t), \boldsymbol{u}_n(t)\big) + \big(\boldsymbol{f}(t),\boldsymbol{u}_n(t)\big),   
\end{align}
$a.  e. ~ on ~ (0,T)$. Adding equations \eqref{11} and \eqref{12}, we have
\begin{align}\label{12a}
 &\dfrac{1}{2}\dfrac{d}{dt}\left(\left\| \boldsymbol{u}_n(t)\right\| _{L^{2}}^{2} + \hat \delta\left\| \boldsymbol{\nabla} C_n(t)\right\| _{L^{2}}^{2}\right)+ \alpha \mu\left\| \boldsymbol{u}_n(t)\right\| _{L^{2}}^{2}+ \left( d\hat \delta - \epsilon \hat \delta \right) \left\|\Delta C_n(t) \right\| _{L^{2}}^{2} + \mu_{e} \left\| \boldsymbol{\nabla} \boldsymbol{u}_n(t)\right\| _{L^{2}}^{2} \nonumber \\ & \leq - \beta \mu \big(C_n(t) \boldsymbol{u}_n(t), \boldsymbol{u}_n(t)\big)+ M(\epsilon) \hat \delta \|g\|_{L^{\infty}(\Omega)}^{2}  \left\|  C_n(t)\right\| _{L^{2}}^{2}  + \big(\boldsymbol{f}(t),\boldsymbol{u}_n(t)\big) ~ ~ a.  e. ~ on ~ (0,T).   
\end{align}
By employing the Gagliardo-Nirenberg inequality (Theorem \ref{gagliardo}), Holder's inequality, and Young's inequality (Remark \ref{youngs}), we derive the following estimate:
\begin{align*}
    \left| - \beta \mu \big(C_n(t) \boldsymbol{u}_n(t), \boldsymbol{u}_n(t)\big) \right| &\leq \beta \mu \|C_n(t) \|_{L^2} \| \boldsymbol{u}_n(t)\|_{L^4}^{2}\\ & \leq M \beta \mu \|C_{0}\|_{L^2} \|\boldsymbol{u}_n(t)\|_{L^2} \|\boldsymbol{\nabla} \boldsymbol{u}_n(t)\|_{L^2}\\ & \leq \epsilon \|\boldsymbol{\nabla} \boldsymbol{u}_n(t)\|_{L^2}^{2} + \beta^2 \mu^2 M^2(\epsilon) \|C_{0}\|_{L^2}^{2} \|\boldsymbol{u}_n(t)\|_{L^2}^{2}  ~ ~ a.  e. ~ on ~ (0,T).
\end{align*}
Also, we have \begin{align*}
    \left| \big(\boldsymbol{f}(t),\boldsymbol{u}_n(t)\big)\right| \leq \frac{1}{2}\|\boldsymbol{f}(t)\|_{L^2}^{2} + \frac{1}{2}\|\boldsymbol{u}_n(t)\|_{L^2}^{2}  ~ ~ a.  e. ~ on ~ (0,T).
\end{align*}
Using the above estimates in equation \eqref{12a}, we arrive at
\begin{align}\label{13}
 &\dfrac{1}{2}\dfrac{d}{dt}\left(\left\| \boldsymbol{u}_n(t)\right\| _{L^{2}}^{2} + \hat \delta\left\| \boldsymbol{\nabla} C_n(t)\right\| _{L^{2}}^{2}\right)+ \alpha \mu\left\| \boldsymbol{u}_n(t)\right\| _{L^{2}}^{2}+\left( d\hat \delta - \epsilon \hat \delta \right)\left\|\Delta C_n(t) \right\| _{L^{2}}^{2} + \mu_{e} \left\| \boldsymbol{\nabla} \boldsymbol{u}_n(t)\right\| _{L^{2}}^{2} \nonumber \\ & \leq \epsilon  \left\| \boldsymbol{\nabla} \boldsymbol{u}_n(t)\right\|_{L^{2}}^{2} + \beta^2 \mu^2  M^{2}(\epsilon) \left\| C_0\right\|_{L^{2}}^{2}\left\| \boldsymbol{u}_n(t)\right\| _{L^{2}}^{2} + \left( \dfrac{1}{2}\left\|\boldsymbol{f}(t)\right\| _{L^{2}}^{2}+\dfrac{1}{2}\left\| \boldsymbol{u}_n(t)\right\| _{L^{2}}^{2}\right)\nonumber \\ &+  M(\epsilon) \hat \delta \|g\|_{L^{\infty}(\Omega)}^{2} \left\|  C_n(t)\right\| _{L^{2}}^{2}  ~ ~ a.  e. ~ on ~ (0,T).   \end{align}
Now we choose $\epsilon>0$ such that $ \epsilon < \min  \left\{\mu_{e}/ 2,  
 d/ 2\right\}$ and ignore all the non-negative terms on the left-hand side of equation \eqref{13}, to obtain
\begin{align*}
    \dfrac{1}{2}\dfrac{d}{dt}\left(\left\| \boldsymbol{u}_n(t)\right\| _{L^{2}}^{2} + \hat \delta\left\| \boldsymbol{\nabla} C_n(t)\right\| _{L^{2}}^{2}\right) \leq \dfrac{1}{2}\left\|\boldsymbol{f}(t)\right\| _{L^{2}}^{2} + M \hat \delta \|g\|_{L^{\infty}(\Omega)}^{2} \left\|  C_n(t)\right\| _{L^{2}}^{2} + \left(\dfrac{1}{2} + \beta^2 \mu^2  M^{2} \left\| C_0\right\|_{L^{2}}^{2} \right)\left\| \boldsymbol{u}_n(t)\right\| _{L^{2}}^{2}
\end{align*}
$a.  e. ~ on ~ (0,T)$. Adding a positive term, $\left(\dfrac{1}{2} + \beta^2 \mu^2  M^{2} \left\| C_0\right\|_{L^{2}}^{2} \right) \hat \delta\left\| \boldsymbol{\nabla} C_n(t)\right\| _{L^{2}}^{2} $ in the RHS of the above inequality and then  integrating each term in time from $0$ to some $\tau \in (0,T)$, we have
\begin{align*}
\left\| \boldsymbol{u}_n(\tau)\right\| _{L^{2}}^{2} + \hat \delta\left\| \boldsymbol{\nabla} C_n(\tau)\right\| _{L^{2}}^{2} &\leq \left\| \boldsymbol{u}_n(0)\right\| _{L^{2}}^{2} + \hat \delta\left\| \boldsymbol{\nabla} C_n(0)\right\| _{L^{2}}^{2} +\int ^{\tau}_{0}\left\|\boldsymbol{f}(t)\right\| _{L^{2}}^{2}\\  &+ 2 M \hat \delta \|g\|_{L^{\infty}(\Omega)}^{2}\int_{0}^{\tau}   \left\|  C_n(t)\right\| _{L^{2}}^{2} + \left(1 + 2\beta^2 \mu^2  M^{2} \left\| C_0\right\|_{L^{2}}^{2}  \right) \int ^{\tau}_{0} \left(\left\| \boldsymbol{u}_n(t)\right\| _{L^{2}}^{2} + \hat \delta\left\| \boldsymbol{\nabla} C_n(t)\right\| _{L^{2}}^{2}\right). 
\end{align*}
Grönwall's inequality (Theorem \ref{GR}) in the above inequality combined with Lemma \ref{le41} proves that the velocity field sequence $\boldsymbol{u}_n$ is uniformly bounded in $ L^{\infty }\left( 0,T;S_{1}\right)$  and the concentration sequence  $C_n$  is uniformly bounded in  $L^{\infty }\left( 0,T;H^1\right)$ space. Now from equation \eqref{13}, we also get the sequence $\boldsymbol{u}_n$ is uniformly bounded in $L^{2 }\left( 0,T;V_{1}\right)$  and the sequence  $C_n$ is uniformly bounded in $L^{2 }\left( 0,T;V_{2}\right)$ space.
\end{proof}

%%%%%%%%%%%%% Lemma 4.3 %%%%%%%%%%%%%%%%%%%%%%%%%%%%%%%%%%%%%%%%%%%%%5
%%%%%%%%%%%%%%%%%%%%%%%%%%%%%%%%%%%%%%%%%%%%%%%%%%%%%%%%%%%%%%%%%%%%%%%%%%%%%%

\begin{lem}\label{le43} The sequence $\frac{\partial C_n}{\partial t}$ is uniformly bounded in  $L^{2}\left( 0,T;L^{2}\right) $ space.
\end{lem}
\begin{proof}
Multiplying equation \eqref{9} by  $ \frac{\partial \beta^{n}_j(t)}{\partial t} $, and summing it from $j=1$ to $n$, we arrive at 
\begin{align*}
\left\| \dfrac{\partial C_n(t)}{\partial t}\right\| _{L^{2}}^2 \leq d\left\| \Delta C_n(t)\right\| _{L^{2}} \left\| \dfrac{\partial C_n(t)}{\partial t}\right\| _{L^{2}} + \left\| \boldsymbol{u}_n(t)\cdot \boldsymbol{\nabla} C_n(t)\right\| _{L^{2}} \left\| \dfrac{\partial C_n(t)}{\partial t}\right\| _{L^{2}} + \left\| C_n(t)g\right\| _{L^{2}} \left\| \dfrac{\partial C_n(t)}{\partial t}\right\| _{L^{2}}.  \end{align*}
\begin{align*}
\left\| \dfrac{\partial C_n(t)}{\partial t}\right\| _{L^{2}} &\leq d\left\| \Delta C_n(t)\right\| _{L^{2}}+\left\| \boldsymbol{u}_n(t)\cdot \boldsymbol{\nabla} C_n(t)\right\| _{L^{2}}+\left\| C_n(t)g\right\| _{L^{2}}\\ &\leq d\left\| \Delta C_n(t)\right\| _{L^{2}}+\left\| \boldsymbol{u}_n(t)\right\| _{L^{4}}\left\| \boldsymbol{\nabla} C_n(t)\right\| _{L^{4}}+\left\| g\right\| _{L^{\infty }}\left\| C_n(t)\right\| _{L^{2}} ~ ~ a.  e. ~ on ~ (0,T).    
\end{align*}
From Gagliardo–Nirenberg inequality (Theorem \ref{gagliardo}), Young's inequality (Remark \ref{youngs}), and the above inequality, we further arrive at
\begin{align*}
\left\| \dfrac{\partial C_n(t)}{\partial t}\right\| _{L^{2}}&\leq d\left\| \Delta C_n(t)\right\| _{L^{2}}+M\left\| \boldsymbol{u}_n(t)\right\| _{L^{2}}^{1/2}\left\|  \boldsymbol{u}_n(t)\right\| _{H^{1}}^{1/2}\left\|  C_n(t)\right\| _{H^{1}}^{1/2}\left\|  C_n(t)\right\| _{H^{2}}^{1/2}+\left\| g\right\| _{L^{\infty }}\left\| C_n(t)\right\| _{L^{2}} \\ & \leq  d\left\|\Delta C_n(t)\right\| _{L^{2}}+M\big(\left\| \boldsymbol{u}_n(t)\right\| _{H^{1}} + \left\| C_n(t)\right\| _{H^{2}}\big)+\left\| g\right\| _{L^{\infty }}\left\| C_n(t)\right\| _{L^{2}} ~ ~ a.  e. ~ on ~ (0,T).    
\end{align*}
Again from Lemma \ref{le42}, we observe all the terms present on the right-hand side of the above inequality are uniformly bounded in $L^2(0,T)$. Therefore we can conclude that $\frac{\partial C_n}{\partial t}$ is uniformly bounded in  $L^{2}\left( 0,T;L^{2}\right)$ space.
\end{proof}

%%%%%%%%%%%%% Lemma 4.4 %%%%%%%%%%%%%%%%%%%%%%%%%%%%%%%%%%%%%%%%%%%%%5
%%%%%%%%%%%%%%%%%%%%%%%%%%%%%%%%%%%%%%%%%%%%%%%%%%%%%%%%%%%%%%%%%%%%%%%%%%%%%%

\begin{lem}\label{le44} The sequence 
$\frac{\partial \boldsymbol{u}_n}{\partial t}$ is uniformly bounded in $L^{2}\left( 0,T;V^{\ast}_{1}\right) $ space.
\end{lem}
\begin{proof}  Multiplying equation \eqref{16} by $\alpha_{j}^{n}$ and summing it from $j=1$ to $n$, and using Holder's inequality, we get
\begin{align}\label{14}
\left\| \dfrac{\partial \boldsymbol{u}_n(t)}{\partial t}\right\| _{V_{1}^{\ast }}\leq \mu_{e}\left\| \boldsymbol{\nabla} \boldsymbol{u}_n(t)\right\| _{L^{2}} + \alpha \mu \left\| \boldsymbol{u}_n(t)\right\| _{L^{2}} + \beta \mu \left\|C_n(t) \boldsymbol{u}_n(t) \right\|_{L^{2}} +\left\| \boldsymbol{\nabla} \cdot \mathbf{T}\left( C_n(t)\right) \right\| _{V_{1}^{\ast }} + \left\| \boldsymbol{f}(t) \right\|_{L^{2}} 
\end{align}
$a.  e. ~ on ~ (0,T)$. First, we estimate the last term on the right side of equation \eqref{14}. Note that $\boldsymbol{\nabla}\cdot \mathbf{T}(C_n)$ is a sum of expressions of the form  `` $D_{i}(D_{j}C_nD_{k}C_n)$" 
\begin{align*}
\left\| D_{i}\left( D_{j}C_n(t)D_{k}C_n(t)\right) \right\| _{V_{1}^{\ast }}\leq \left\| D_{j}C_n(t)D_{k}C_n(t)\right\| _{L^{2}} \leq \left\| D_{j}C_n(t)\right\| _{L^{4}}\left\| D_{k}C_n(t)\right\| _{L^{4}} ~ ~ a.  e. ~ on ~ (0,T).  \end{align*}
Using  Gagliardo–Nirenberg inequality (Theorem \ref{gagliardo}) and Young's inequality (Remark \ref{youngs}), we have
\begin{align*}
\left\| D_{i}\left( DjC_n(t)D_{k}C_n(t)\right) \right\| _{V_{1}^{\ast }}
&\leq M\left\| C_n(t)\right\| _{H^{1}}^{1/2}\left\| C_n(t)\right\| _{H^{2}}^{1/2}\left\|C_n(t)\right\| _{H^{1}}^{1/2}\left\| C_n(t)\right\| _{H^{2}}^{1/2} ~ ~ a.  e. ~ on ~ (0,T).   
\end{align*}
Again applying  Gagliardo–Nirenberg inequality (Theorem \ref{gagliardo}) and Young's inequality (Remark \ref{youngs}) gives us
 \begin{align*}
 \left\| C_n(t) \boldsymbol{u}_n(t)  \right\|_{L^2} \leq  \left\| C_n(t) \right\|_{L^4} \left\| \boldsymbol{u}_n(t)  \right\|_{L^4}  &\leq M \left\| C_n(t)   \right\|_{L^2}^\frac{1}{2} \left\| C_n(t)   \right\|_{H^1}^\frac{1}{2} \left\|  \boldsymbol{u}_n(t)  \right\|_{L^2}^\frac{1}{2} \left\|  \boldsymbol{u}_n(t)  \right\|_{H^1}^\frac{1}{2}\\ &\leq M \left( \left\| C_n(t)   \right\|_{H^1}  + \left\| \boldsymbol{u}_n(t)   \right\|_{H^1}    \right)~ ~  a.  e. ~ on ~ (0, T).
 \end{align*}
 Using both of these estimates and results from lemmas \ref{le41}-\ref{le42} in equation \eqref{14}, we observe that the sequence $\frac{\partial \boldsymbol{u}_n}{\partial t}$ is bounded in $L^{2}(0,T,V_{1}^{\ast})$ space. 
\end{proof}

\textbf{Proof of the Theorem 4.1.}  From lemmas \ref{le41}-\ref{le42}, we establish that
  $\boldsymbol{u}_{n}  $ and
	$C_n $  are uniformly bounded in     $ L^{\infty}\left(0,T;S_{1}\right)$   and $ L^{\infty}\left(0,T;H^1\right)$  respectively, $\forall~ T>0$.
 Consequently, employing continuation argument, we conclude that  $\boldsymbol{u}_{n}(t) $ and  $C_{n}(t)$  exist for all  $t \in (0,T)$. Additionally, based on Lemmas \ref{le41}-\ref{le44}, we can deduce the existence of a constant  $M > 0$ such that:
 $$\left\| C_n   \right\|_{L^{\infty}\left(0,T;H^1\right)} ,
\left\| C_n   \right\|_{L^{2}\left(0,T;V_{2}\right)},
\left\| \boldsymbol{u}_n   \right\|_{L^{\infty}\left(0,T;S_{1}\right)} ,\left\|\boldsymbol{u}_n   \right\|_{L^{2}\left(0,T;V_{1}\right)}, \left\| C_{n}'  \right\|_{L^{2}\left(0,T;L^2\right)}, \left\|\boldsymbol{u}_n'  \right\|_{L^{2}\left(0,T;V_1^{*}\right)} \leq M,$$
for all $T>0$.

By applying the diagonalization argument to the above estimates, we obtain sub-sequences which are again labelled by $C_n$ and $\boldsymbol{u}_n$, and those satisfy the following convergence result:

\begin{center}
$\begin{aligned}
C_n \rightarrow C \hspace{15pt} &  \text{weakly} &&\text{in }  L^{2}\left(0,T;V_2\right), \\
C_n \rightarrow C \hspace{15pt} & \text{weakly-}*  && \text{in }  L^{\infty}\left(0,T;H^1\right), \\
C_n'\rightarrow C'\hspace{12pt} &\text{weakly} && \text{in }  L^{2}\left(0,T;V^{\ast}_{2}\right), \\
\boldsymbol{u}_n \rightarrow \boldsymbol{u} \hspace{15pt}&\text{weakly} &&\text{in } L^{2}\left(0,T;V_1\right), \\
\boldsymbol{u}_n \rightarrow \boldsymbol{u} \hspace{15pt}&\text{weakly-}* &&\text{in }  L^{\infty}\left(0,T;S_{1}  \right), \\
\boldsymbol{u}_n'\rightarrow \boldsymbol{u}'\hspace{12pt} &\text{weakly} &&\text{in } L^{2}\left(0,T;V^{\ast}_{1}\right), 
\end{aligned}$
\end{center}

\paragraph{\textbf{Passing to the limit}}For passing the limit into equation \eqref{15}, we fix a $m \in \mathbb{N}$ such that $m \leq n$. Then for all $B \in (V_2)_m$, we have 
\begin{align}\label{l1}
 \left (\dfrac{\partial C_{n}(t)}{\partial t},B\right ) +d (\boldsymbol{\nabla} C_{n}(t) , \boldsymbol{\nabla} B) + (\boldsymbol{u}_n(t)\cdot \boldsymbol{\nabla} C_{n}(t),B)+ (g C_n(t), B)=0 ,~ ~  a.  e. ~ on ~ (0,T).   
\end{align}
Since $B$ is a linear combination of $z_j$ for $j \in \{1,2,3....m  \}$.
Now using $C_n \rightarrow C $ weakly in  $L^2(0,T;V_2)$ and $\frac{\partial C_n}{\partial t} \rightarrow \frac{\partial C}{\partial t} $ weakly in  $L^2(0,T;V^{*}_2)$, we see that
\begin{align}\label{l2}
 \lim_{n\rightarrow \infty} \left(\boldsymbol{\nabla} C_n,\boldsymbol{\nabla}B  \right)  &=  \left(\boldsymbol{\nabla} C,\boldsymbol{\nabla}B  \right),~ \forall B \in (V_2)_m ~ ~  a.  e. ~ on ~ (0,T),   \end{align}
 and
 \begin{align}\label{l5}
  \lim_{n\rightarrow \infty} \left(\frac{\partial C_n}{\partial t}, B\right) &= \left(\frac{\partial C_n}{\partial t}, B  \right),~ \forall B \in (V_2)_m ~ ~  a.  e. ~ on ~ (0,T).   
\end{align}
Using $g \in L^{\infty}(\Omega)$, and $C_n \rightarrow C $  in  $L^2(0,T;H^1)$,   we have 
\begin{align}\label{15a}
  \lim_{n\rightarrow \infty} \left(g C_n - g C, B\right) \leq \left\|g\right\|_{L^{\infty}(\Omega)} \lim_{n\rightarrow \infty}\|C_n - C \|_{L^2} \| B\|_{L^2}  = 0, ~ \forall B \in (V_2)_m ~ ~  a.  e. ~ on ~ (0,T).   \end{align}
Now we take the limit for the convection term in equation \eqref{15}. 
\begin{align}\label{17}
  \int_{\Omega} \boldsymbol{u}_n(t)\cdot \boldsymbol{\nabla} C_n(t) B - \boldsymbol{u}(t)\cdot \boldsymbol{\nabla} C(t) B &=  \int_{\Omega}\Big( \boldsymbol{u}_n(t)\cdot \boldsymbol{\nabla} C_n(t) -  \boldsymbol{u}_n(t)\cdot \boldsymbol{\nabla} C(t) + \boldsymbol{u}_n(t)\cdot \boldsymbol{\nabla} C(t) - \boldsymbol{u}(t)\cdot \boldsymbol{\nabla} C(t)\Big) B \nonumber\\ & \leq \left\| \boldsymbol{\nabla}C_n(t)-\boldsymbol{\nabla} C(t)  \right\|_{L^2} \left\|\boldsymbol{u}_n(t) B  \right\|_{L^2} + \left\|\boldsymbol{u}_n(t) - \boldsymbol{u}(t) \right\|_{L^2} \left\|\boldsymbol{\nabla}C(t)  B  \right\|_{L^2}
\end{align}
$a.  e. ~ on ~ (0, T)$. From the strong convergence of $C_n$ in $L^2(0,T;H^1)$ we have $\left\| \boldsymbol{\nabla}C_n(t)-\boldsymbol{\nabla} C(t)  \right\|_{L^2} \rightarrow 0 $ as $n \rightarrow \infty$ $a.  e. ~ on ~ (0,T)$. Similarly the strong convergence of $\boldsymbol{u}_n$ in $L^2(0,T;S_1)$ implies $\left\| \boldsymbol{u}_n(t)-\boldsymbol{u}(t)  \right\|_{L^2} \rightarrow 0 $ as $n \rightarrow \infty$ $a.  e. ~ on ~ (0,T)$. Having these convergence results and the observation that $\left\|\boldsymbol{u}_n(t) B  \right\|_{L^2}$ and $\left\|\boldsymbol{\nabla}C_{n}(t)  B  \right\|_{L^2}$ are bounded independent of $n$, we get
\begin{align}\label{l3}
\lim_{n\rightarrow \infty} \big( \boldsymbol{u}_n(t)\cdot \boldsymbol{\nabla} C_n(t), B   \big)=
      \big(\boldsymbol{u}(t)\cdot \boldsymbol{\nabla} C(t), B\big) ~ \quad a.  e. ~ on ~ (0,T)~ \forall ~ B \in (V_2)_m.
\end{align}
As we choose $m \in \mathbb{N}$ arbitrary, thus convergence results in equation \eqref{l2}, \eqref{l5}, \eqref{15a} and \eqref{l3} follows for all $B \in \cup_{m \ge 1} (V_2)_m $. Using the fact that $\cup_{m \ge 1} (V_{2})_m$ is dense in $V_2$, we observe that convergence results \eqref{l2}, \eqref{l5}, \eqref{15a}, and \eqref{l3} hold for all $B \in V_2$.
Thus using these convergence results from equation \eqref{l2}, \eqref{l5}, \eqref{15a}, and \eqref{l3}, passing the limit in equation \eqref{l1}, we conclude 
\begin{align*}
 \left (\dfrac{\partial C(t)}{\partial t},B\right ) +d (\boldsymbol{\nabla} C(t) , \boldsymbol{\nabla} B) + (\boldsymbol{u}(t)\cdot \boldsymbol{\nabla} C(t),B) + (g C(t), B)=0 ,~ \forall\, B \in V_2 ~ ~  a.  e. ~ on ~ (0,T).   
\end{align*}
Again for passing the limit in equation \eqref{16}, we fix a $m\in \mathbb{N}$ such that $m \leq n$. Then for all $B \in (V_1)_m$, we have
\begin{align}\label{l4}
\innerproduct*{\dfrac{\partial \boldsymbol{u}_{n}(t)}{\partial t},\boldsymbol{v}} +\mu_{e}(\boldsymbol{\nabla} \boldsymbol{u}_{n}(t) , \boldsymbol{\nabla} \boldsymbol{v}) +\alpha \mu ( \boldsymbol{u}_{n}(t), \boldsymbol{v}) + \beta \mu \left( C_n(t)\boldsymbol{u}_n(t), \boldsymbol{v}   \right)= \innerproduct{\boldsymbol{\nabla} \cdot \mathbf{T}\left( C_{n}(t)\right), \boldsymbol{v}} + \left(\boldsymbol{f}(t),\boldsymbol{v}\right),
\end{align}
$a.  e. ~ on ~ (0, T)$. Since $\boldsymbol{u}_n \rightarrow \boldsymbol{u}$ weakly in $L^2(0,T;V_1)$ and $\frac{\partial \boldsymbol{u}_n}{\partial t} \rightarrow \frac{\partial \boldsymbol{u}}{\partial t}$ weakly in $L^2(0,T;V^{*}_1)$, we can pass the limit in all the terms present on the left-hand side of equation \eqref{l4} except last one. We will pass the limit in non-linear terms one by one. We first pass the limit in the term present on the right-hand side of equation \eqref{l4}. Using the expression of $\innerproduct*{\boldsymbol{\nabla} \cdot \mathbf{T}\left( C_{n}(t)\right), \boldsymbol{v}}$ from equation \eqref{4}, we have
\begin{align*}
 \innerproduct*{\boldsymbol{\nabla} \cdot \mathbf{T}\left( C_{n}(t)\right), \boldsymbol{v}} = -\hat{\delta}  \int_{\Omega} \Delta C_n(t) \boldsymbol{\nabla} C_n(t) \cdot \boldsymbol{v} \quad ~ a.  e. ~ on ~ (0,T).
\end{align*}
We can further write each component of  $\Delta C_n\boldsymbol{\nabla} C_n \text{ by } \sum_{{i,j,k} \in \{1,2\}} a_{ijk} D_i\left(D_j C_n D_k C_n \right)$ for some constants $a_{ijk} \in \mathbb{R}$ and $i,j, k \in \{1,2\}$. Using integration by parts, we have
\begin{align}\label{18}
    -\hat{\delta} \sum_{{i,j,k} \in \{1,2\}} a_{ijk} \int_{\Omega} D_i\left(D_j C_{n}(t) D_k C_{n}(t)\right) {v} = \hat{\delta} \sum_{{i,j,k} \in \{1,2\}} a_{ijk} \int_{\Omega} \left(D_j C_{n}(t) D_k C_{n}(t)\right) D_i{v}
\end{align}
$a.  e. ~ on ~ (0,T)$. Now from Gagliardo–Nirenberg inequality (Theorem \ref{gagliardo}), we have
\begin{align*}
    \|  D_j C_n(t) D_k C_n(t) \|_{L^{2}} \leq M \| C_n(t) \|_{H^1}^{\frac{1}{2}} \| C_n(t) \|_{H^2}^{\frac{1}{2}} \|C_n(t)  \|_{H^1}^{\frac{1}{2}} \| C_n(t) \|_{H^2}^{\frac{1}{2}}  \quad ~ a.  e. ~ on ~ (0,T),
\end{align*}
and from Lemmas \ref{le41} -\ref{le44}, we know that all the terms present on the right-hand side of the above inequality are uniformly bounded in $L^2(0,T)$. This means that $D_j C_n(t) D_k C_n(t)$ is also uniformly bounded in $L^2(\Omega)  ~ a.  e. ~ on ~ (0,T).$ Using reflexive weak compactness, we obtain a weakly convergent subsequence $D_j C_n D_k C_n$ (the notation is unchanged for convenience) in $L^2(\Omega)$. Without loss of generality, we consider this weak limit as $D_j C D_k C$. Now this weak convergence of  $D_j C_n D_k C_n$ in $L^2(\Omega)  ~ a.  e. ~ on ~ (0,T)$ implies
\begin{align*}
    \lim_{n \rightarrow \infty} \int_{\Omega} \left(D_j C_{n}(t) D_k C_{n}(t)\right) D_i{v} = \int_{\Omega} \left(D_j C(t) D_k C(t)\right) D_i{v} \quad ~ a.  e. ~ on ~ (0,T).
\end{align*}
Utilizing this convergence result, we conclude
\begin{align*}
 \lim_{n \rightarrow \infty}\innerproduct*{\boldsymbol{\nabla} \cdot \mathbf{T}\left( C_{n}(t)\right), \boldsymbol{v}} = \innerproduct*{\boldsymbol{\nabla} \cdot \mathbf{T}\left( C(t)\right), \boldsymbol{v}}  \quad ~ a.  e. ~ on ~ (0,T) ~ \forall ~ \boldsymbol{v} \in (V_1)_m.
\end{align*}
Again we have,
\begin{align*}
   \big(C_n(t) \boldsymbol{u}_n(t) - C(t) \boldsymbol{u}(t) , \boldsymbol{v} \big) &=   \int_{\Omega} \big(C_n(t) \boldsymbol{u}_n(t) - C(t) \boldsymbol{u}(t)\big)\cdot \boldsymbol{v} \\ & = \int_{\Omega} \big(C_n(t) \boldsymbol{u}_n(t) - C_n(t) \boldsymbol{u}(t) + C_n(t) \boldsymbol{u}(t) - C(t) \boldsymbol{u}(t)\big)\cdot \boldsymbol{v} \\ & = \int_{\Omega} C_n(t) \big(\boldsymbol{u}_n(t) - \boldsymbol{u}(t)\big)\cdot \boldsymbol{v} + \int_{\Omega} \big(C_n(t) - C(t)\big)\boldsymbol{u}(t)\cdot\boldsymbol{v} \\ & \leq \left\|\boldsymbol{u}_n(t) - \boldsymbol{u}(t) \right\|_{L^2}  \left\| C_n(t) \boldsymbol{v}\right\|_{L^2} + \left\|C_n(t)-C(t)\right\|_{L^2} \left\|\boldsymbol{u}(t)\cdot\boldsymbol{v}\right\|_{L^2} 
  \\ & \leq \left\|\boldsymbol{u}_n(t) - \boldsymbol{u}(t) \right\|_{L^2}  \left\| C_n(t) \right\|_{L^4}\left\| \boldsymbol{v} \right\|_{L^4} + \left\|C_n(t)-C(t)\right\|_{L^2} \left\|\boldsymbol{u}(t)\right\|_{L^4}\left\| \boldsymbol{v} \right\|_{L^4} ~ ~  a.  e. ~ on ~ (0,T).
\end{align*}
Using Gagliardo-Nirenberg inequality (Theorem \ref{gagliardo}) in the above inequality, we have 
\begin{align*}&\big(C_n(t) \boldsymbol{u}_n(t) - C(t)\boldsymbol{u}(t)  , \boldsymbol{v} \big)\\ &\leq
   M \left\|\boldsymbol{u}_n(t) - \boldsymbol{u}(t) \right\|_{L^2}  \left\| C_n(t) \right\|_{L^2}^{\frac{1}{2}}\left\| C_n(t) \right\|_{H^1}^{\frac{1}{2}}\left\| \boldsymbol{v} \right\|_{L^2}^{\frac{1}{2}}\left\| \boldsymbol{v} \right\|_{H^1}^{\frac{1}{2}} + M \left\|C_n(t)-C(t)\right\|_{L^2} \left\|\boldsymbol{u}\right\|_{L^2}^{\frac{1}{2}}\left\|\boldsymbol{u}\right\|_{H^1}^{\frac{1}{2}}\left\| \boldsymbol{v} \right\|_{L^2}^{\frac{1}{2}} \left\| \boldsymbol{v} \right\|_{H^1}^{\frac{1}{2}},
\end{align*}
$a.  e. ~ on ~ (0, T)$. Using the result that $\left\|\boldsymbol{u}_n(t) - \boldsymbol{u}(t) \right\|_{L^2} \rightarrow 0 $ and $\left\|C_n(t)-C(t)\right\|_{L^2} \rightarrow 0 $ as $n \rightarrow \infty ~ ~  a.  e. ~ on ~ (0,T)$ and all the other terms are uniformly bounded in the above inequality, we conclude
\begin{align*}
    \lim_{n \rightarrow \infty} \big(C_n(t) \boldsymbol{u}_n(t),\boldsymbol{v}\big) = \big(C(t) \boldsymbol{u}(t), \boldsymbol{v}\big) ~  ~ a.  e. ~ on ~ (0,T) \text{ and } \forall ~ \boldsymbol{v} \in (V_1)_m.
\end{align*}
Again using the fact that $\cup_{m \ge 1} (V_1)_m$ is dense in $V_1$, we pass the limit in equation \eqref{l4}, we have 
\begin{align*}
\innerproduct*{\dfrac{\partial \boldsymbol{u}(t)}{\partial t},\boldsymbol{v}} +\mu_{e}(\boldsymbol{\nabla} \boldsymbol{u}(t) , \boldsymbol{\nabla} \boldsymbol{v}) + \mu\big((\alpha + \beta C(t))\boldsymbol{u}(t), \boldsymbol{v}\big) = \innerproduct{\boldsymbol{\nabla} \cdot \mathbf{T}\left( C(t)\right), \boldsymbol{v}} +(\boldsymbol{f}(t),\boldsymbol{v}) ,~\forall\,\boldsymbol{v} \in V_{1} ~ ~  a.  e. ~ on ~ (0,T).
\end{align*}
With this, we have proved that there exists a velocity field and concentration that satisfies the variational form given by equations \eqref{9}-\eqref{10}. In order to establish the existence of a pressure solution, we will utilize de Rham's (Theorem \ref{de rham}). Let

$$ \boldsymbol{h} = -\frac{\partial \boldsymbol{u}}{\partial t} - \mu (\alpha + \beta C)\boldsymbol{u}+ \mu_{e} \Delta \boldsymbol{u} + \boldsymbol{\nabla}\cdot \mathbf{T}(C)+\boldsymbol{f},$$
then 
$$\innerproduct*{\boldsymbol{h}, \boldsymbol{\phi}}= -\innerproduct*{\frac{\partial \boldsymbol{u}}{\partial t}, \boldsymbol{\phi}} - \mu \big((\alpha + \beta C)\boldsymbol{u},\boldsymbol{\phi}\big) - \mu_{e}\left( \boldsymbol{\nabla}\boldsymbol{u}, \boldsymbol{\nabla} \phi\right) + \innerproduct*{\boldsymbol{\nabla}\cdot \mathbf{T}(C),\boldsymbol{\phi}}+(\boldsymbol{f},\boldsymbol{\phi}), ~ \forall \boldsymbol{\phi} \in \mathbb{D}(\Omega).   $$
Now from equation \eqref{10}, we have $\innerproduct*{\boldsymbol{h}, \boldsymbol{\phi}} = 0, ~ \forall  \boldsymbol{\phi} \in \mathbb{D}(\Omega)$,
so by de Rham's (Theorem \ref{de rham}) there exist a function $p \in L^2(\Omega)$ such that $ \boldsymbol{h} = \boldsymbol{\nabla} p $.
This completes the proof of existence part of the Theorem \ref{th41}.
%%%%%%%%%%%%%%%%%% Uniqueness %%%%%%%%%%%%%%%%%%%%%%%%%%%%%%%
%%%%%%%%%%%%%%%%%%%%%%%%%%%%%%%%%%%%%%%%%%%%%%%%%%%%%%%%%%%%%%
%%%%%%%%%%%%%%%%%%%%%%%%%%%%%%%%%%%%%%%%%%%%%%%%%%%%%%%%%%%%%%
\paragraph{\textbf{Uniqueness of Solution}} 
Let $(\boldsymbol{u}_{1},C_{1},p_{1}) \text{ and } (\boldsymbol{u}_{2},C_{2},p_{2})$ be two solutions of the problem \eqref{1}-\eqref{6} in the sense of Definition \ref{def31}. If we put these solutions in equation \eqref{9} and subtract, then we have
\begin{align}\label{22a}
\int _{\Omega} \dfrac{\partial C(t)}{\partial t}B - d\int _{\Omega}\Delta C(t) B +\int _{\Omega}\left( \boldsymbol{u}_{1}(t)\cdot\boldsymbol{\nabla} C_{1}(t)-\boldsymbol{u}_{2}(t)\cdot \boldsymbol{\nabla}C_{2}(t)\right)B + \int_{\Omega}  C(t) g B = 0, ~\forall~ B \in V_2.    
\end{align}
$ a.  e. ~ on ~ (0,T)$, where $C=C_1-C_2$. Using $B =  C(t)$ in the above equation, and utilizing integration by parts and boundary conditions \eqref{5}, we obtain
\begin{align*}
&\frac{1}{2}\frac{d}{dt}\|C(t) \|_{L^2}^2
 + d\|\boldsymbol{\nabla} C(t)\|_{L^2}^2 \nonumber \\ & = -\int _{\Omega}\big( \boldsymbol{u}_{1}(t)\cdot\boldsymbol{\nabla} C_{1}(t) -  \boldsymbol{u}_{2}(t)\cdot\boldsymbol{\nabla} C_{1}(t)+  \boldsymbol{u}_{2}(t)\cdot\boldsymbol{\nabla} C_{1}(t) - \boldsymbol{u}_{2}(t)\cdot \boldsymbol{\nabla}C_{2}(t)\big)C(t) - \int_{\Omega}  C(t) g C(t)  ~ ~ a.  e. ~ on ~ (0,T).    
\end{align*}
\begin{align}\label{22b}
    \frac{1}{2}\frac{d}{dt}\|C(t) \|_{L^2}^2
 + d\|\boldsymbol{\nabla} C(t)\|_{L^2}^2 = -\int _{\Omega} \boldsymbol{u}(t)\cdot\boldsymbol{\nabla} C_{1}(t) C(t) - \int _{\Omega} \boldsymbol{u}_2(t)\cdot\boldsymbol{\nabla} C(t) C(t) - \int_{\Omega}  C(t) g C(t)  ~ ~ a.  e. ~ on ~ (0,T),
\end{align}
where $\boldsymbol{u}= \boldsymbol{u}_1- \boldsymbol{u}_2$. Now, the second term on the right-hand side above the equation will vanish as shown in lemma \ref{le41}. We estimate the first term in the right-hand side of the above equation using Holder's inequality, Gagliardo-Nirenberg inequality (Theorem \ref{gagliardo}), and Young's inequality (Remark \ref{youngs}), as the following: 
\begin{align*}
  \left| \int _{\Omega} \boldsymbol{u}(t)\cdot\boldsymbol{\nabla} C_{1}(t) C(t)    \right|  &\leq \|C(t)\|_{L^2} \| \boldsymbol{u}(t)\cdot\boldsymbol{\nabla} C_{1}(t) \|_{L^2}\\ &\leq \|C(t)\|_{L^2} \|\boldsymbol{u}(t)\|_{L^4} \|\boldsymbol{\nabla} C_{1}(t) \|_{L^4}\\ & \leq M \|C(t)\|_{L^2}\|\boldsymbol{u}(t)\|_{L^2}^{1/2} \|\boldsymbol{\nabla}\boldsymbol{u}(t)\|_{L^2}^{1/2} \|\boldsymbol{\nabla} C_{1}(t) \|_{L^2}^{1/2} \|\Delta C_{1}(t) \|_{L^2}^{1/2} \\ & \leq M \|C(t)\|_{L^2}^2 + \epsilon \|\boldsymbol{\nabla}\boldsymbol{u}(t)\|_{L^2}^{2} + M(\epsilon)\|\boldsymbol{u}(t)\|_{L^2}^{2} \|\boldsymbol{\nabla} C_{1}(t) \|_{L^2}^{2} \|\Delta C_{1}(t) \|_{L^2}^{2}
\end{align*}
$a.  e. ~ on ~ (0,T)$. Using this estimate in equation \eqref{22b}, we have 
\begin{align}\label{22c}
    \frac{1}{2}\frac{d}{dt}\|C(t) \|_{L^2}^2
 + d\|\boldsymbol{\nabla} C(t)\|_{L^2}^2 \leq \epsilon \|\boldsymbol{\nabla}\boldsymbol{u}(t)\|_{L^2}^{2} + M(\epsilon)\|\boldsymbol{u}(t)\|_{L^2}^{2} \|\boldsymbol{\nabla} C_{1}(t) \|_{L^2}^{2} \|\Delta C_{1}(t) \|_{L^2}^{2} + \big(M + \|g\|_{L^{\infty}(\Omega)}\big) \|C(t)\|_{L^2}^2
\end{align}
$a.  e. ~ on ~ (0,T)$.
Using ~$B=-\hat \delta\Delta C(t)$ in equation \eqref{22a}, we obtain
\begin{align*}
 &\int _{\Omega} \dfrac{\partial C(t) }{\partial t}(-\hat \delta\Delta C(t)) -d\int _{\Omega}\Delta C(t)(-\hat \delta\Delta C(t))+\int_{\Omega}  C(t) g (-\hat \delta\Delta C(t))\\ & = -\int _{\Omega}\big( \boldsymbol{u}_{1}(t)\cdot\boldsymbol{\nabla} C_{1}(t)-\boldsymbol{u}_{2}(t)\cdot \boldsymbol{\nabla} C_{2}(t)\big)(-\hat \delta\Delta C(t))   ~ ~ a.  e. ~ on ~ (0,T). 
\end{align*}
Integration by parts and boundary conditions \eqref{5} further gives us 
\begin{align*}
\hat \delta &\int _{\Omega}\dfrac{\partial(\boldsymbol{\nabla} C(t)) }{\partial t}\cdot \boldsymbol{\nabla} C(t) +d \hat \delta\int _{\Omega}\Delta C(t)\Delta C(t) + \hat \delta\int _{\Omega}\boldsymbol{u}_{2}(t)\cdot \boldsymbol{\nabla} C_{2}(t)\Delta C(t) -  \hat \delta \int _{\Omega}C(t) g \Delta C(t)\\ & = \hat \delta\int _{\Omega}\big(\boldsymbol{u}_{1}(t)\cdot \boldsymbol{\nabla} C_{1}(t)\Delta C(t)+ \boldsymbol{u}_{1}(t)\cdot \boldsymbol{\nabla} C_{2}(t)\Delta C(t)-\boldsymbol{u}_{1}(t)\cdot \boldsymbol{\nabla} C_{2}(t)\Delta C(t)\big)  ~ ~ a.  e. ~ on ~ (0,T).     
\end{align*}
Rearrangements of the terms of the above equation lead us to
\begin{align*}
\dfrac{ \hat \delta}{2}&\dfrac{d}{dt}\left\| \boldsymbol{\nabla} C(t)\right\| _{L^{2}}^{2}+d \hat \delta\left\| \Delta C(t)\right\| _{L^{2}}^{2}  \\&=  \hat \delta\int _{\Omega}\boldsymbol{u}_{1}(t)\cdot \boldsymbol{\nabla} C(t) \Delta C(t) + \hat \delta\int _{\Omega} \boldsymbol{u}(t)\cdot  \boldsymbol{\nabla} C_{2}(t)\Delta C(t)  + \hat \delta \int _{\Omega}C(t)g\Delta C(t) ~ ~ a.  e. ~ on ~ (0, T).
\end{align*}
Now as in lemma \ref{le42}, using the fact that $g \in L^{\infty}(\Omega)$  with Holder's and Young's inequality (Remark \ref{youngs}), on the term associated with $g$ in the right-hand side of the above equality, we arrive at
\begin{align}\label{22}
\dfrac{ \hat \delta}{2}\dfrac{d}{dt}\left\| \boldsymbol{\nabla} C(t)\right\| _{L^{2}}^{2}+\left(d \hat \delta  - \epsilon \hat \delta\right)\left\| \Delta C(t)\right\| _{L^{2}}^{2} &\leq  \hat \delta \int _{\Omega}\boldsymbol{u}_{1}(t)\cdot \boldsymbol{\nabla}  C(t) \Delta C(t) +  \hat \delta\int _{\Omega}\boldsymbol{u}(t)\cdot  \boldsymbol{\nabla} C_{2}\Delta C(t)\nonumber \\& + \hat \delta  M(\epsilon)  \|g\|_{L^{\infty}(\Omega)}^{2} \left\|  C(t)\right\| _{L^{2}}^{2} ~ ~ ~ a.  e. ~ on ~ (0, T).
\end{align}
Putting $(\boldsymbol{u}_1, C_1)$ and $(\boldsymbol{u}_2, C_2)$ in equation \eqref{10} and subtracting, we have
\begin{align*}
    &\innerproduct*{\dfrac{\partial \boldsymbol{u}(t) }{\partial t} ,\boldsymbol{v}} - \mu_{e}\left(\boldsymbol{\nabla}\boldsymbol{u}(t),\boldsymbol{\nabla}\boldsymbol{v}\right)+\alpha \mu \left( \boldsymbol{u}(t),\boldsymbol{v}  \right)\\& = \innerproduct{\boldsymbol{\nabla}\cdot \mathbf{T}(C_1(t)) -   \boldsymbol{\nabla}\cdot \mathbf{T}(C_2(t)) \big), \boldsymbol{v} } -  \beta \mu \big(\left(C_{1}(t)\boldsymbol{u}_{1}(t)-C_{2}(t)\boldsymbol{u}_{2}(t)\right), \boldsymbol{v}\big) ~ ~a.  e. ~ on ~ (0,T).
\end{align*}
Putting  expression of ``$\innerproduct{\boldsymbol{\nabla}\cdot \mathbf{T}(C(t)),\boldsymbol{v}}$" from equation \eqref{4}, and $\boldsymbol{v} = \boldsymbol{u}(t)$, in above equality, we conclude
\begin{align}\label{23}
 &\dfrac{1}{2}\dfrac{d}{dt}\left\| \boldsymbol{u}(t)\right\| _{L^{2}}^{2}+\mu_{e}\left\| \boldsymbol{\nabla} \boldsymbol{u}(t)\right\| _{L^{2}}^{2}+\alpha \mu\left\| \boldsymbol{u}(t)\right\| _{L^{2}}^{2}  +   \hat \delta \int _{\Omega}\Delta C_{1}(t)\boldsymbol{\nabla} C(t)\cdot \boldsymbol{u}(t) \nonumber \\& = -\hat \delta \int _{\Omega}\Delta C(t)\boldsymbol{\nabla} C_{2}(t)\cdot \boldsymbol{u}(t) -   \beta \mu \int_{\Omega}\left(C_{1}(t)\boldsymbol{u}_{1}(t)-C_{1}(t)\boldsymbol{u}_{2}(t)+C_{1}(t)\boldsymbol{u}_{2}(t)-C_{2}(t)\boldsymbol{u}_{2}(t)\right)\boldsymbol{u}(t).   
\end{align}
$a.  e. ~ on ~ (0,T)$. Then the summation of equations \eqref{22c}, \eqref{22} and \eqref{23} implies
\begin{align}\label{24}
&\dfrac{1}{2}\dfrac{d}{dt}\left(\left\| \boldsymbol{u}(t)\right\| _{L^{2}}^{2} + \|C(t)\|_{L^2}^2 +\hat \delta\left\| \boldsymbol{\nabla} C(t)\right\| _{L^{2}}^{2}\right)+\mu_{e}\left\| \boldsymbol{\nabla} \boldsymbol{u}(t)\right\| _{L^{2}}^{2}+\alpha \mu\left\| \boldsymbol{u}(t)\right\| _{L^{2}}^{2}  +  d \|\boldsymbol{\nabla}C(t) \|_{L^2}^2 + (d\hat \delta - \epsilon \hat \delta)\left\| \Delta C(t)\right\| _{L^{2}}^{2}\nonumber\\& \leq \hat \delta\int _{\Omega}\boldsymbol{u}_{1}(t)\cdot \boldsymbol{\nabla}  C(t) \Delta C(t)  -\hat \delta \int _{\Omega}\Delta C_{1}(t)\boldsymbol{\nabla} C(t)\cdot \boldsymbol{u}(t) \nonumber -   \beta \mu \int _{\Omega}\left(\left(C_{1}(t)\boldsymbol{u}(t)+C(t)\boldsymbol{u}_{2}(t)\right)\cdot\boldsymbol{u}(t)\right)  +  M(\epsilon)\hat \delta \|g\|_{L^{\infty}(\Omega)}^{2} \left\|  C(t)\right\| _{L^{2}}^{2}\nonumber \\ &  \epsilon \| \boldsymbol{\nabla} \boldsymbol{u} \|_{L^2}^2 + M(\epsilon)\|\boldsymbol{u}(t)\|_{L^2}^{2} \|\boldsymbol{\nabla} C_{1}(t) \|_{L^2}^{2} \|\Delta C_{1}(t) \|_{L^2}^{2} + \big(M + \|g\|_{L^{\infty}(\Omega)}\big) \|C(t)\|_{L^2}^2  ~ ~ a.  e. ~ on ~ (0,T).
\end{align}
First we estimate the terms present on the right-hand side of equation \eqref{24}. The use of Holder inequality, Gagliardo-Nirenberg inequality (Theorem \ref{gagliardo}), and Young's inequality (Remark \ref{youngs}) leads us to the following inequality
\begin{align}\label{25}
\left|\int _{\Omega} \Delta C_{1}(t)\boldsymbol{\nabla} C(t)\cdot \boldsymbol{u}(t)\right| &\leq \left\| \Delta C_{1}(t)\right\| _{L^{2}}\left\|  \boldsymbol{\nabla} C(t) \cdot \boldsymbol{u}(t)  \right\| _{L^{2}} \nonumber \\ &\leq M\left\| \Delta C_1(t)\right\| _{L^{2}}\left\| \boldsymbol{\nabla} C(t)\right\| _{L^{2}}^{1/2}\left\| \Delta C(t)\right\| _{L^{2}}^{1/2}\left\| \boldsymbol{u}(t)\right\| _{L^{2}}^{1/2}\left\| \boldsymbol{\nabla} \boldsymbol{u}(t)\right\| _{L^{2}}^{1/2}\nonumber\\ & \leq 2 \epsilon \big(\| \Delta C(t)\|_{L^2}\|\boldsymbol{\nabla}\boldsymbol{u} \|_{L^2}\big) + M(\epsilon) \| \Delta C_{1}(t)\|_{L^2}^{2} \|\boldsymbol{u}\|_{L^2} \|\boldsymbol{\nabla}C(t) \|_{L^2} \nonumber\\ &\leq  \epsilon \left(\left\| \boldsymbol{\nabla}\boldsymbol{u}(t)\right\| _{L^{2}}^{2}+\left\| \Delta C(t)\right\| _{L^{2}}^{2}\right) + M(\epsilon)\left\| \Delta C_{1}(t)\right\| _{L^{2}}^{2}(\left\| \boldsymbol{\nabla} C(t)\right\| _{L^{2}}^2 + \left\|  \boldsymbol{u}(t)\right\| _{L^{2}}^2)    
\end{align}
$a.  e. ~ on ~ (0,T)$.
Again using Holder inequality, Gagliardo-Nirenberg inequality (Theorem \ref{gagliardo}), and Young's inequality (Remark \ref{youngs}), we get the following estimate
\begin{align}\label{26}
 \left|\int _{\Omega} \boldsymbol{u}_{1}(t)\cdot \boldsymbol{\nabla} C(t)\Delta C(t)\right| &\leq \left\| \Delta C(t)\right\| _{L^{2}}\left\| \boldsymbol{\nabla} C(t)\cdot \boldsymbol{u}_{1}(t)\right\| _{L^{2}} \leq \left\| \Delta C(t)\right\| _{L^{2}}\left\| \boldsymbol{\nabla} C(t)\right\| _{L^{4}}\left\| \boldsymbol{u}_{1}(t)\right\| _{L^{4}} \nonumber\\ &\leq M\left\| \Delta C(t)\right\| _{L^{2}}\left\| \boldsymbol{\nabla} C(t)\right\| _{L^{2}}^{1/2}\left\| \Delta C(t)\right\| _{L^{2}}^{1/2}\left\| \boldsymbol{u}_{1}(t)\right\| _{L^{2}}^{1/2}\left\| \boldsymbol{\nabla} \boldsymbol{u}_{1}(t)\right\| _{L^{2}}^{1/2}\nonumber\\ & \leq \epsilon \left\| \Delta C(t)\right\| _{L^{2}}^{2}+ M(\epsilon)\left\| \boldsymbol{\nabla} C(t)\right\| _{L^{2}}^{2}\left\| \boldsymbol{u}_{1}(t)\right\| _{L^{2}}^{2}\left\| \boldsymbol{\nabla} \boldsymbol{u}_{1}(t)\right\| _{L^{2}}^{2} ~ ~ a.  e. ~ on ~ (0,T).   
\end{align}
Similarly, we have,
\begin{align}\label{27}
  \left|\int _{\Omega} C_{1}(t) \boldsymbol{u}(t)\cdot \boldsymbol{u}(t)\right| & \leq \left\|C_{1}(t)\right\| _{L^{2}}  \left\|\boldsymbol{u}(t)\cdot \boldsymbol{u}(t)\right\| _{L^{2}} \nonumber\\ &\leq   \epsilon\left\| \boldsymbol{\nabla} \boldsymbol{u}(t)\right\| _{L^{2}}^{2}+M(\epsilon)\left\|  \boldsymbol{u}(t)\right\| _{L^{2}}^{2} ~ ~ a.  e. ~ on ~ (0,T),     
\end{align}
and
\begin{align}\label{28}
 \left|\int _{\Omega} C(t) \boldsymbol{u}_{2}(t)\cdot \boldsymbol{u}(t)\right|& \leq \left\|C(t)\right\| _{L^{2}}  \left\| \boldsymbol{u}_{2}(t)\cdot \boldsymbol{u}(t)\right\| _{L^{2}}\nonumber\\ &\leq    \left\|C(t)\right\| _{L^{2}} \left\| \boldsymbol{u}_{2}(t)\right\| _{L^{4}}\left\|\boldsymbol{u}(t)\right\| _{L^{4}}\nonumber\\ &\leq M\|C(t) \|_{L^2}^2 + \epsilon \left\| \boldsymbol{\nabla} \boldsymbol{u}(t)\right\| _{L^{2}}^{2} + M(\epsilon) \left\|  \boldsymbol{u}_{2}(t)\right\| _{L^{2}}^2 \left\| \boldsymbol{\nabla} \boldsymbol{u}_{2}(t)\right\| _{L^{2}}^2 \left\|  \boldsymbol{u}(t)\right\| _{L^{2}}^2 ~ ~ a.  e. ~ on ~ (0,T).   
\end{align}
Then the estimates from equations \eqref{25}, \eqref{26}, \eqref{27}, and \eqref{28} in equation \eqref{24} is used to obtain the following inequality,
\begin{align*}
&\dfrac{1}{2}\dfrac{d}{dt}\left(\left\| \boldsymbol{u}(t)\right\| _{L^{2}}^{2}+\|C(t)\|_{L^2}^2+\hat \delta\left\| \boldsymbol{\nabla} C(t)\right\| _{L^{2}}^{2}\right)+\big(\mu_{e}- \epsilon(1+\hat \delta +2\beta \mu)\big)\left\| \boldsymbol{\nabla} \boldsymbol{u}(t)\right\| _{L^{2}}^{2}+\alpha \mu\left\| \boldsymbol{u}(t)\right\| _{L^{2}}^{2} + d \|\boldsymbol{\nabla}C_n(t)\|_{L^2}^2 \\ &+ \hat \delta\big(d - 3\epsilon\big)\left\| \Delta C(t)\right\| _{L^{2}}^{2} \leq  M(\epsilon) \hat \delta \Big( \left\| \Delta C_{1}(t)\right\| _{L^{2}}^{2}(\left\| \boldsymbol{\nabla} C(t)\right\| _{L^{2}}^2 + \left\|  \boldsymbol{u}(t)\right\| _{L^{2}}^2)  +\left\| \boldsymbol{\nabla} C(t)\right\| _{L^{2}}^{2}\left\| \boldsymbol{u}_{1}(t)\right\| _{L^{2}}^{2}\left\| \boldsymbol{\nabla} \boldsymbol{u}_{1}(t)\right\| _{L^{2}}^{2}\Big)\\& +\beta \mu M(\epsilon)\Big( \left\|  \boldsymbol{u}(t)\right\| _{L^{2}}^{2} +  \left\|  \boldsymbol{u}_{2}(t)\right\| _{L^{2}}^2 \left\| \boldsymbol{\nabla} \boldsymbol{u}_{2}(t)\right\| _{L^{2}}^2 \left\|  \boldsymbol{u}(t)\right\| _{L^{2}}^2 \Big)  + M(\epsilon)\|\boldsymbol{u}(t)\|_{L^2}^{2} \|\boldsymbol{\nabla} C_{1}(t) \|_{L^2}^{2} \|\Delta C_{1}(t) \|_{L^2}^{2} \\ &+   \Big( M\big(1+\beta \mu\big)+\big(1+\hat \delta M(\epsilon)\big)\|g\|_{L^{\infty}(\Omega)}^{2})\Big)  \left\|  C(t)\right\| _{L^{2}}^{2}   ~ ~ a.  e. ~ on ~ (0,T).  
\end{align*}
 Now if we choose $\epsilon>0$ such that  $0 < \epsilon < \min\left(\frac{\mu_e}{(1+\hat \delta +2\beta \mu)},\frac{d}{3}\right)$ and neglect all the non-negative terms on the left-hand side of  the above inequality, then we have
\begin{align}\label{29}
\dfrac{1}{2}\dfrac{d}{dt}\left(\left\| \boldsymbol{u}(t)\right\| _{L^{2}}^{2}+\|C(t)\|_{L^2}^2+\hat \delta\left\| \boldsymbol{\nabla} C(t)\right\| _{L^{2}}^{2}\right)  \leq M  \phi(t)\left(\left\| \boldsymbol{u}(t)\right\| _{L^{2}}^{2}+\|C(t)\|_{L^2}^2+\hat \delta\left\| \boldsymbol{\nabla} C(t)\right\| _{L^{2}}^{2}\right) ~ ~ a.  e. ~ on ~ (0,T).   
\end{align}
Where \begin{align*}
    \phi(t) & =   \left\| \Delta C_{1}(t)\right\| _{L^{2}}^{2} + \left\| \boldsymbol{u}_{1}(t)\right\| _{L^{2}}^{2}\left\| \boldsymbol{\nabla} \boldsymbol{u}_{1}(t)\right\| _{L^{2}}^{2} + \beta \mu +  \beta \mu \left\|  \boldsymbol{u}_{2}(t)\right\| _{L^{2}}^2 \left\| \boldsymbol{\nabla} \boldsymbol{u}_{2}(t)\right\| _{L^{2}}^2\\& +  \|\boldsymbol{\nabla} C_{1}(t) \|_{L^2}^{2} \|\Delta C_{1}(t) \|_{L^2}^{2}  +  \big(1+\beta \mu\big)+\big(1/M+\hat \delta \big)\|g\|_{L^{\infty}(\Omega)}^{2} ~ ~ a.  e. ~ on ~ (0,T).
\end{align*}
Integrating both side of equation \eqref{29} in time from 0 to some $\tau \in (0,T)$, we have
\begin{align*}
 \int ^{\tau}_{0}\dfrac{d}{dt}\left( \left\| \boldsymbol{u}(t)\right\| _{L^{2}}^{2}+\|C(t)\|_{L^2}^2 +\hat \delta \left\| \boldsymbol{\nabla} C(t)\right\| _{L^{2}}^{2}\right) \leq M\int ^{\tau}_{0} \phi(t) \left( \left\| \boldsymbol{u}(t)\right\| _{L^{2}}^{2}+ \|C(t)\|_{L^2}^2 +\hat \delta \left\| \boldsymbol{\nabla} C(t)\right\| _{L^{2}}^{2}\right). 
\end{align*}
This implies,
\begin{align*}
 \int ^{\tau}_{0}\dfrac{d}{dt}\left( \exp\left(-M\int ^{t}_{0}\phi \left( s\right) ds \right)  \left( \left\| \boldsymbol{u}(t)\right\| _{L^{2}}^{2}+\|C(t)\|_{L^2}^2+\hat \delta\left\| \boldsymbol{\nabla} C(t)\right\| _{L^{2}}^{2}\right)\right)\leq 0.   
\end{align*}
Which means,
\begin{align*}
\exp\left(-M\int ^{\tau}_{0}\phi \left( s\right) ds\right) \left( \left\| \boldsymbol{u}(\tau)\right\| _{L^{2}}^{2}+\|C(t)\|_{L^2}^2+\hat \delta \left\| \boldsymbol{\nabla} C(\tau)\right\| _{L^{2}}^2\right)  \leq \left\| \boldsymbol{u}\left( 0\right) \right\| _{L^{2}}^{2}+\|C(0)\|_{L^2}^2+\hat \delta \left\| \boldsymbol{\nabla} C\left( 0\right) \right\| _{L^{2}}^{2}\quad \forall~ \tau \in (0,T).    
\end{align*}
Now, since we know that $\boldsymbol{u}(0)=0$ and $C(0)=0$ from the standard setting, thus the above inequality proves the uniqueness of the solution for the problem given by \eqref{1}-\eqref{4} with boundary and initial data given by equation \eqref{5} and equation \eqref{6}, respectively. Moreover, the uniqueness of $p$ is ensured by the given constraint on $p$, i.e., $\int p =0$.\\
This completes the proof.

By setting $\beta = 0$ in the mathematical model defined by equations \eqref{1}-\eqref{6}, we derive a formulation commonly employed to depict flows in homogeneous reactive porous media. Consequently, we ascertain its well-posedness by demonstrating the same property established for the model discussed in Section \ref{sec:math_model}.

Moreover, when $\boldsymbol{f}=\boldsymbol{0}$ and $g=0$ in the problem described by equations \eqref{1}-\eqref{6}, we obtain a model representing non-reactive heterogeneous porous media flow devoid of external forces. The existence and uniqueness of its solution are straightforwardly established.

Again if we set $\beta = 0$, $\boldsymbol{f} = \boldsymbol{0}$, and $g = 0$ all together, then the domain in equations \eqref{1}-\eqref{6} can be replaced as $(t, x) \in$ $(0, \infty) \times \Omega$ where its well-posedness can be shown. This simplified model describes a homogeneous, non-reactive flow without any imposed external force. For this specific model, the following result holds.
\begin{corollary}\label{cor411}
    If we set $\beta = 0$, $\boldsymbol{f} = \boldsymbol{0}$, and $g = 0$ in problem defined by equations \eqref{1}-\eqref{6}, then for any initial condition $ (\boldsymbol{u}_{0},C_{0}) \in (S_{1},H^1) $, there exists a solution pair $(\boldsymbol{u},C)$ of the modified problem in domain $(t, x) \in$ $(0, \infty) \times \Omega$. Furthermore, we have
	$ \boldsymbol{u} \in L^{2}(0,\infty;V_{1}) \cap \mathcal{C}(\mathbb{R}^{+},S_{1})$  and $ C \in  L^{2}(0,\infty;V_{2}) \cap \mathcal{C}(\mathbb{R}^{+},H^{1}).$
\end{corollary}
\begin{proof} From lemma \ref{le41}, we have
\begin{align}\label{31}
\left\| C_n\left(\tau\right)\right\| _{L^{2}\left(\Omega \right) }^{2} + 2 d \int_{0}^{\tau}{\left\| \boldsymbol{\nabla} C_n(t)\right\|}^{2} _{L^{2}\left( \Omega\right) } \leq \left\| C_n\left(0\right)\right\| _{L^{2}\left(\Omega \right) }^{2}\leq \left\| C(0)\right\| _{L^{2}(\Omega)}^2,\hspace{5pt} \forall \tau > 0.     
\end{align}
By  choosing $\beta = 0, g=0, \boldsymbol{f} = \boldsymbol{0}, \text{ and } \alpha = \dfrac{1}{K}$ in equation \eqref{12a}, we conclude 
\begin{align*}
 &\dfrac{1}{2}\dfrac{d}{dt}\left(\left\| \boldsymbol{u}_n(t)\right\| _{L^{2}}^{2} + \hat \delta\left\| \boldsymbol{\nabla} C_n(t)\right\| _{L^{2}}^{2}\right)+ \dfrac{\mu}{K}\left\| \boldsymbol{u}_n(t)\right\| _{L^{2}}^{2}+d\hat \delta\left\|\Delta C_n(t) \right\| _{L^{2}}^{2} + \mu_{e} \left\| \boldsymbol{\nabla} \boldsymbol{u}_n(t)\right\| _{L^{2}}^{2} \leq 0 ~ ~ a.  e. ~ on ~ (0,T).   
\end{align*}
Integrating the above equality in time from $0$ to $\tau \in (0,T]$, we obtain
\begin{align}\label{32}
 &\dfrac{1}{2}\left(\left\| \boldsymbol{u}_n(t)\right\| _{L^{2}}^{2} + \hat \delta\left\| \boldsymbol{\nabla} C_n(t)\right\| _{L^{2}}^{2}\right)+ \dfrac{\mu}{K}\int_{0}^{\tau}\left\| \boldsymbol{u}_n(t)\right\| _{L^{2}}^{2}+d\hat \delta \int_{0}^{\tau}\left\|\Delta C_n(t) \right\| _{L^{2}}^{2} + \mu_{e}\int_{0}^{\tau} \left\| \boldsymbol{\nabla} \boldsymbol{u}_n(t)\right\| _{L^{2}}^{2} \nonumber\\ &\leq \dfrac{1}{2}\left(\left\| \boldsymbol{u}_n(0)\right\| _{L^{2}}^{2} + \hat \delta\left\| \boldsymbol{\nabla} C_n(0)\right\| _{L^{2}}^{2}\right) ,\hspace{5pt} \forall \tau > 0.  
 \end{align}
 From inequalities \eqref{31} and \eqref{32}, we have $\boldsymbol{u}_n$ uniformly bounded in  $ L^{2}(0,\infty;V_{1}) \cap L^{\infty}(0,\infty;S_{1})$  and $C_n$ is uniformly bounded in $ L^{2}(0,\infty;V_{2}) \cap  L^{\infty}(0,\infty;H^{1}) $. This fact, combined with all other previously shown results (since they remain valid $ \forall ~ T>0$) imply that, $\boldsymbol{u} \in L^{2}(0,\infty;V_{1}) ~\cap~ \mathcal{C}(\mathbb{R}^{+},S_{1})$  and $ C \in  L^{2}(0,\infty;V_{2}) ~\cap~ \mathcal{C}(\mathbb{R}^{+},H^{1})$.
\end{proof}

\section{Concluding remarks}
In this paper, we have demonstrated the existence and uniqueness of a weak solution to the unsteady Darcy-Brinkman problem in the presence of an external force field coupled with a convection-diffusion-reaction equation (incorporating a simple first-order reaction). The heterogeneity of the porous media is also taken into account by considering permeability as an inverse linear function of solute concentration. The validity of the proofs remains straightforward when considering non-reactive homogeneous porous media flow without any body forces, achieved by setting appropriate parameters to zero. One of the more challenging tasks for future research is to establish the well-posedness of the problem when the permeability exponentially decreases with concentration, as noted in \cite{Nagatsu2014}. Additionally, a more realistic approach would be to address the problem with variable fluid viscosity and higher-order reactions, as suggested in \cite{anne2020}.

\section*{Acknowledgments}
S.K. acknowledges the support of UGC, Government of India, with a research fellowship.

\bibliographystyle{plain}
\bibliography{references}
\end{document}